\documentclass[12pt,a4paper,onecolumn]{amsart}

\title[Large solutions for subordinate spectral]{Large solutions for subordinate spectral Laplacian}

\author{Ivan Bio\v{c}i\'{c} and Vanja Wagner}
\date{}

\usepackage[utf8]{inputenc}
\usepackage{fullpage}

\usepackage{amsmath}
\usepackage{amsthm}
\usepackage{amsfonts}
\usepackage{mathtools}
\usepackage[dvipsnames]{xcolor} 
\usepackage{amssymb} 
\usepackage{enumerate}
\usepackage{comment}
\usepackage{aligned-overset}
\usepackage{yfonts}
\usepackage{cancel,xcolor,soul}

\usepackage{hyperref}
\hypersetup{
	colorlinks=true,
	linkcolor=magenta,
}
\DeclareMathOperator\supp{supp}

\newtheorem{thm}{Theorem}[section]

\newtheorem{cor}[thm]{Corollary}

\newtheorem{lem}[thm]{Lemma}
\theoremstyle{definition}
\newtheorem{rem}[thm]{Remark}

\newtheorem*{assumption*}{\assumptionnumber}
\providecommand{\assumptionnumber}{}
\makeatletter
\newenvironment{assumption}[2]
{%
	\renewcommand{\assumptionnumber}{(\textbf{#1#2})}%
	\begin{assumption*}%
		\protected@edef\@currentlabel{(\textbf{#1#2})}%
	}
	{%
	\end{assumption*}
}
\makeatother

\newcommand{\R}{\mathbb{R}}
\newcommand{\N}{\mathbb{N}}
\newcommand{\p}{\mathbb{P}}
\newcommand{\ex}{\mathbb{E}}
\newcommand{\subsub}{\subset\subset}
\newcommand{\1}{\mathbf 1}

\newcommand{\LLL}{L^1(D,\de(x)dx)}
\newcommand{\BB}{\mathcal{B}}

\newcommand{\KK}{\mathcal{K}}
\newcommand{\DD}{\mathcal{D}}

\newcommand{\Lo}{\phi(-\left.\Delta\right\vert_{D})}
\newcommand{\Loz}{\phi^*(-\left.\Delta\right\vert_{D})}

\newcommand{\wLo}{\phi_d(-\left.\Delta\right\vert_{D})}

\newcommand{\LoR}{\phi(-\Delta)}

\newcommand{\wLoR}{\phi_d(-\Delta)}

\newcommand{\LozR}{\phi^*(-\Delta)}

\newcommand{\GDfi}{G^\phi_{ D}}
\newcommand{\GDFI}{\GDfi}
\newcommand{\GDfz}{G^{\phi^*}_{ D}}
\newcommand{\GD}{G_{ D}}
\newcommand{\PDfi}{P^\phi_{ D}}
\newcommand{\PDFI}{\PDfi}

\newcommand{\de}{\delta_D}
\newcommand{\wh}{\widehat}
\newcommand{\wt}{\widetilde}
\newcommand{\diam}{\textrm{diam}}
\newcommand{\dist}{\textrm{dist}}
\newcommand{\uu}{\textswab{u}}
\newcommand{\vv}{\textswab{v}}

\numberwithin{equation}{section}

\begin{document}

\begin{abstract}
    We find a large solution to a semilinear Dirichlet problem in a bounded $C^{1,1}$ domain for a non-local operator $\phi(-\Delta\vert_{D})$, an extension of the infinitesimal generator of a subordinate killed Brownian motion. The setting covers and extends the case of the spectral fractional Laplacian. The upper bound for the explosion rate of the large solution is obtained and is given in terms of the renewal function, distance to the boundary, and the Keller-Osserman-type transformation of the nonlinearity. 
    Additionally, we prove interior higher regularity results for this operator.
\end{abstract}
\maketitle


\section{Introduction}

Let $D\subset \R^d$, $d\ge3$, be a $C^{1,1}$ bounded domain and $f:\R\to [0,\infty)$. In this article, we find a solution to the following semilinear problem
\begin{equation}\label{eq:problem_intro}
	\begin{array}{rcll}
		\Lo u&=& - f(u)& \quad \text{in } D,\\
		\frac{u}{\PDFI\sigma}&=&\infty&\quad \text{on }\partial D.
	\end{array}
\end{equation}
Here the operator $\Lo$ is an extension of the infinitesimal generator of a subordinate killed Brownian motion, where the subordinator has the Laplace exponent $\phi$, and $\PDFI\sigma$ is a reference function defined as the Poisson potential of the $d-1$ dimensional Hausdorff measure $\sigma$  on $\partial D$ for the subordinate killed  Brownian motion. The solution to \eqref{eq:problem_intro} is called a large solution since it cannot be uniformly bounded by a nonnegative $\Lo$-harmonic function.

The operator $\Lo$ is a non-local operator of the elliptic type, and for a suitably regular function $u$ can be written as the principal value integral
\begin{align*}
    \Lo u(x)=\textrm{P.V.}\int_D\big(u(x)-u(y)\big)J_D(x,y)dy+\kappa(x)u(x), \quad x \in D,
\end{align*}
where the singular jumping kernel $J_D$ and the so-called killing function $\kappa$ are completely determined by $\phi$, see \eqref{e:J_D}. Also, $\Lo$ can be written in the spectral form 
\begin{align*}
    \Lo u=\sum_{j=1}^\infty \phi(\lambda_j)\wh u_j\varphi_j,
\end{align*}
where $(\varphi_j,\lambda_j)_j$ are the eigenpairs of  $-\left.\Delta\right\vert_{D}$ -- the Dirichlet Laplacian in $D$. 

A classical example of an operator of this form is the spectral fractional Laplacian $(-\left.\Delta\right\vert_{D})^{\alpha/2}$, where $\phi(\lambda)=\lambda^{\alpha/2}$, for $\alpha\in(0,2)$, i.e.~the subordinator is the $\alpha$-stable subordinator.

In this article we make the first step towards the general theory of non-local operators, moving beyond the spectral fractional Laplacian, by considering functions $\phi$ which are complete Bernstein functions satisfying the weak scaling condition at infinity: There exist $a_1,a_2>0$ and $\delta_1, \delta_2  \in(0,1)$ satisfying
\begin{align*}\label{WSC_intro}
a_1\lambda^{\delta_1}\phi(t)\le \phi(\lambda t)\le a_2\lambda^{\delta_2}\phi(t),\quad t\ge1, \,\lambda\ge 1.\tag{WSC}
\end{align*}
This assumption covers a large class of subordinators and extends the classical stable, tempered stable, regularly varying classes (for a list of half a dozen examples we refer to \cite[p. 41]{vondra_heat}), and it is a standard assumption in many state-of-the-art results covering a wide range of general non-local theory, e.g. \cite{semilinear_bvw,BL-na21,BMS23,CKKW22,CKSV22,GKK20,KSV-jde23} and the references therein.

The large problem \eqref{eq:problem_intro} is solved both in the distributional and the pointwise sense, where the boundary condition $\frac{u}{\PDFI\sigma}=\infty$ stands for the limit at the boundary of $u/\PDFI\sigma$ both in the pointwise as well as in the weak $L^1$ sense, described in \eqref{eq:boundary Dirichlet limit} and \eqref{eq:distri solution boundary}. It is a characteristic of non-local settings that boundary blow-up of solutions is possible even for linear equations. For example, the $\Lo$-harmonic in $D$ reference function $\PDFI\sigma$ explodes at the boundary $\partial D$. Therefore, in the non-local setting, the interpretation of a large solution is that it exhibits a greater blow-up at the boundary than any harmonic function for $\Lo$. In fractional settings, the blow-up rates are written as the fractional powers of the distance to the boundary. In this article, the corresponding blow-up rate can be expressed in terms of the Laplace exponent $\phi$ and the corresponding renewal function.

The analysis of non-local operators and semilinear equations has attracted much attention in recent years, particularly in the case of the fractional Laplacian. Research has predominantly been focused on moderate solutions to semilinear equations for the fractional Laplacian, which are solutions bounded by some nonnegative harmonic function, see \cite{Aba15a,BC17,BCBF16, BCBF18,bogdan_et_al_19,CFQ,FQ12, FQT12}. However, there is also great interest in studying large solutions to semilinear equations, i.e.~the solutions that explode at the boundary with a higher rate than harmonic functions. In the case of fractional Laplacian, such solutions were explored in \cite{Aba17, BCBF16, CFQ}. In recent years, there has been further development in the theory of linear and semilinear problems for more general non-local operators{;} here we refer the reader to \cite{bogdan_extension,BL-na21,BMS23,FJ23,remarks_on_nonlocal,KKLL} for results on linear equations, and to \cite{semilinear_bvw,BJ20,CGCV,HuynhNguyen2022_new,klimsiak2023dirichlet} for moderate solutions to semilinear problems. To the best of our knowledge, the article \cite{semilinear_cvw} is the first to deal with large solutions for non-local operators that generalise the fractional Laplacian.

On the other hand, an interesting modification of the standard semilinear problem involving the fractional Laplacian is to consider other types of stable operators - such as the spectral fractional Laplacian and other regional or spectral-type non-local operators. Here we mention \cite{dhifli2012,grubb2016,SV2014} where homogeneous equations were studied in the case of the spectral fractional Laplacian. Furthermore, we highlight \cite{AbaDupaNonhomo2017}, where the authors introduced the nonhomogeneous boundary condition for the case of the spectral fractional Laplacian and where both moderate and large solutions to semilinear equations were studied.

Motivated by \cite{AbaDupaNonhomo2017}, the case of moderate solutions to nonhomogeneous equations for more general non-local operators of the spectral type was studied in \cite{Bio23}, which serves as a technical cornerstone for this article. The setting here is equivalent to that in \cite{Bio23}, with the only difference that the results in the present work do not cover the case $d=2$. This is due to some constraints in potential-theoretical techniques applied in the paper, which we further comment on below.

Let us now describe the main results of the article. In Section \ref{s:large} we present the central result - the existence of a large solution to \eqref{eq:problem_intro}, Theorem \ref{t:large solution}. For the nonlinearity $f$ we assume that $f\in C^1(\R^d)$, $f(0)=0$, $f(t)>0$ for $t>0$, and
\begin{assumption}{F}{}\label{F_intro}
	There exist $0<m\le M<\infty$ such that
	$$ (1+m)f(t)\le tf'(t)\le (1+M)f(t),\quad t\in \R.$$
\end{assumption}
\noindent The novelty here is that this class of functions somewhat extends the class of power functions $f(t)=t^p$, $p>1$, for which a large solution for the spectral Laplacian was constructed in \cite{AbaDupaNonhomo2017}. In the non-fractional setting of this paper, this construction happens to be more technically difficult, and we obtain it by building on the results in \cite{AbaDupaNonhomo2017,semilinear_cvw,Aba17}. First, motivated by the potential-theoretic properties of the renewal function of the corresponding subordinate Brownian motion, we construct a supersolution to \eqref{eq:problem_intro}. This is done by establishing the following Keller-Osserman-type condition: 
 \begin{equation*}
\int_r^\infty \frac{dt}{\phi^{-1}(\varphi(t)^{-2})}\lesssim\frac{r}{\phi^{-1}(\varphi(r)^{-2})},\qquad r\ge 1,
\end{equation*}
which is sufficient for this construction. Here the function $\varphi$ completely depends on $f$, so in other words the Keller-Osserman-type condition gives a relation between $f$ and $\phi$ under which a large solution to \eqref{eq:problem_intro} exists, for details see Subsection \ref{ss:construction large}. Then, under the additional assumption which ensures the right boundary behaviour of our supersolution, see \eqref{eq:bdry_condition}, we find a large solution to \eqref{eq:problem_intro} by constructing a sequence of approximating moderate solutions. At the end of the section, we prove two notable results. In Theorem \ref{t:minimal large sol} we prove that the obtained large solution is the minimal solution to \eqref{eq:problem_intro}, and in Theorem \ref{t:maximal large sol}, we prove that for every supersolution to \eqref{eq:problem_intro}, there exists the maximal solution dominated by it.

In order to show the desired boundary blow-up of the constructed large solution to \eqref{eq:problem_intro}, we also prove a higher regularity theorem in H\"older spaces for distributional solutions for the operator $\Lo$, as well as for the operator $\LoR$. The latter is the infinitesimal generator of the subordinate Brownian motion and is a direct generalisation of the fractional Laplacian. In Theorem \ref{t:regularity SKBM} we prove that for $f\in C^\alpha(D)$, every distributional solution to 
\[
\Lo u=f\text{ in } D
\]
is in $C^{\alpha+2\delta_1}(D)$ and the usual inequality between the appropriate norms of $f$ and $u$ holds. The detailed proof of this result generalizes the analogous higher regularity approach for the spectral fractional Laplacian given in \cite[Lemma 21]{AbaDupaNonhomo2017}, as well as emphasizes the dimension constraint $d\ge 3$. This limitation seems to be also present, although not explicitly stated in \cite{AbaDupaNonhomo2017,Silvestre2007regularity}, and is necessary for the applied potential-theoretical approach. For details, we refer to Remarks \ref{r:Gf indefinite no1} and \ref{r:Gf indefinite no2}. Consequently, the large solution results obtained in this paper and \cite{AbaDupaNonhomo2017} hold only for $d\ge3$.

One of the most important ingredients in the proof of Theorem \ref{t:regularity SKBM} is a higher regularity theorem for $\LoR$, proved in Theorem \ref{t:regularity SBM}: if $f\in C^\alpha(\R)$, then every distributional solution to $\LoR u=f$ in $\R^d$ is in $C^{\alpha+2\delta_1}(\R)$. The proof of this theorem is motivated by the celebrated article \cite{Silvestre2007regularity}. The approach in the proof of the higher regularity theorem for the fractional Laplacian, \cite[Proposition 2.8]{Silvestre2007regularity}, again exhibits the dimensional constraint $d\ge 3$, which we highlight in the proof of Theorem \ref{t:regularity SBM}. However, this problem was overcome in the fractional setting, see e.g. \cite{Ros-OtonSerra2016}, by using an alternative method on a larger class of operators.

There are many generalisations of \cite[Proposition 2.8]{Silvestre2007regularity} considering generalised H\"older spaces which are more suitable for the development of regularity theory for non-local operators more general than the fractional Laplacian, e.g.~in \cite{remarks_on_nonlocal,Ros-OtonSerra2016,KKLL,KimLee}. The downside of these approaches is that these regularity results usually have an index gap. The major contribution of Theorem \ref{t:regularity SBM} is that the regularity, although obtained in classical H\"older spaces, holds for all Matuszewska indices $0<\delta_1\le\delta_2<1$. We point out that similar results to ours exist in generalised H\"older spaces in \cite{BaeKassmann,KimLee} but they do not cover the whole span of indices $\delta_1$ and $\delta_2$, since their techniques require that indices $\delta_1$ and $\delta_2$ stay close together. It is worth mentioning that, under this assumption on the Matuszewska indices, one can prove a higher regularity claim for $\Lo$ in terms of the generalised H\"older spaces, similar to the one in Theorem \ref{t:regularity SKBM}. Nevertheless, our primary objective is to find a large solution for the operator $\Lo$ for the whole span of Matuszewska indices, and the classical H\"older regularity in Theorem \ref{t:regularity SKBM} suits this purpose well.

The article is organized as follows. At the end of this introductory section, we cover the basic notation used throughout the article. Section \ref{s:prelims} is the preliminary section that deals with the consequences of the assumption \eqref{WSC_intro}, recalling some recent results on the Green kernel and the Poisson kernel of $\Lo$. We also give the precise connection of the operator $\Lo$ with the subordinate killed Brownian motion, and the connection of the operator $\LoR$ with the subordinate Brownian motion, as well as their pointwise representations. We finish this section by recalling the known but delicate differences in different types of solutions and boundary conditions for the Dirichlet problem \eqref{eq:problem_intro}. In particular, in Lemma \ref{l:large solution harmonic unbounded} we prove that a solution of \eqref{eq:problem_intro} is not bounded by any nonnegative harmonic function. In addition, we note the definition and properties of the renewal function $V$, the prominent function in the study of boundary behaviour in $C^{1,1}$ domains. In Section \ref{s:regularity} we present the regularity theory for the operators $\LoR$ and $\Lo$  and prove the previously highlighted higher regularity Theorems \ref{t:regularity SBM} and \ref{t:regularity SKBM}. In the last part, Section \ref{s:large}, we find a large solution to \eqref{eq:problem_intro}. First, we build an auxiliary sequence of solutions to moderate problems. Then, we construct a large supersolution in Subsection \ref{ss:construction large} and find the appropriate Keller-Osserman-type condition. Finally, in Theorem \ref{t:large solution} we prove the existence of a large solution to \eqref{eq:problem_intro}.

\subsection*{Notation}
For an open set $D$, $U\subsub D$ means that $U$ is a nonempty, bounded and open set such that $U\subset \overline U\subset D$, where $\overline U$ denotes the closure of $U$. The boundary of $D$ is denoted by $\partial D$, and $|x|$ denotes the Euclidean norm in $\R^d$. For sets $A,B\subset \R^d$, let $\dist(A,B)\coloneqq \inf\{|x-y|:x\in A,\,y\in B\}$, $\diam D\coloneqq \sup\{|x-y|:x,y\in D\}$, and $\de(x)\coloneqq \dist(x,\partial D)$. By $B(x,r)$ we denote the ball centered at $x\in \R^d$ with radius $r>0$.

We use the standard notation $C(D)$ to denote the space of continuous functions in $D$, $C_b(D)$ the space of bounded functions in $D$, and $C_0(D)$ the space of continuous functions in $D$ that vanish at infinity, i.e.~$f\in C_0(D)$ if for every $\varepsilon>0$ there exists a compact subset $K\subset D$ such that $|f(x)|<\varepsilon$, $x\in D\setminus K$. With $C^k(D)$, $k\ge 1$, we denote all $k$-times differentiable functions in $D$ and $C^\infty(D)$ denotes all infinitely differentiable functions in $D$, where $C_c^\infty(D)$ denotes those with compact support. For $k\in \N_0$ and $\alpha\in (0,1]$, $C^{k,\alpha}(D)$ denotes the space of $k$-times differentiable functions whose all $k$-th derivatives are H\"older continuous in $D$ with the exponent $\alpha$. If $D=\R^d$, we abbreviate the notation by writing, e.g., $C^\infty$ instead of $C^\infty(\R^d)$.

By default, all functions in the article are assumed to be Borel functions, and the space of nonnegative Borel functions in $D$ is denoted by $\BB_+(D)$. The space of $L^p$-integrable functions with respect to the measure $\mu$ is denoted by $L^p(D,\mu)$, where we leave out the set $D$ if $D=\R^d$, and we leave out the measure $\mu$ in the case of the Lebesgue measure on $\R^d$. In the case of local integrability, we write $L_{loc}^p(D,\mu)$. We write $L^\infty(D)$ ($L_{loc}^\infty(D)$) for the set of all (locally) bounded functions in $D$. The set $H_0^1(D)$ is the closure of $C_c^\infty(D)$ in the Sobolev space $H^1(D)=W^{1,2}(D)$ (with respect to the Sobolev norm).

By $\|\cdot\|_{C^{k,\alpha}(D)}$ we denote the standard H\"older norm in $D$, namely
$$\|f\|_{C^{k,\alpha}(D)}=\sum_{j=0}^k \| D^j f\|_{L^\infty(D)}+\sup_{x,y\in D}\frac{|D^kf(x)-D^kf(y)|}{|x-y|^{\alpha}},$$
where the last term we often abbreviate by $[f]_{k,\alpha,D}$.
When there is no reason for confusion, we also abbreviate the standard $L^\infty(D)$ norm by $\|\cdot\|_{L^\infty(D)}=\|\cdot\|_\infty$.

Unimportant constants in the article are denoted by small letters $c$, $c_1$, $c_2$, \dots. Their labeling starts anew in each new statement, and their value may change from line to line. By $C$ we denote more important constants, where $C(a,b)$ means that the constant $C$ depends only on the parameters $a$ and $b$. Naturally, all constants are positive real numbers. We set $a\vee b\coloneqq \max\{a,b\}$ and $a\wedge b \coloneqq \min \{a,b\}$. Finally, for positive functions $f$ and $g$, we write $f\lesssim g$ ($f\gtrsim g$, $f\asymp g$) if there exists a constant $c>0$ such that $f\le c\, g$ ($f\ge c \,g$, $c^{-1} g\le f \le c\,g$).

\section{Preliminaries}\label{s:prelims}
	Let $(W_t)_{t\ge0}$ be a Brownian motion in $\R^d$, $d\ge3$, with the characteristic exponent $\xi\mapsto |\xi|^2$, $\xi\in\R^d$. Let $D$ be a non-empty open set, and $\tau_D\coloneqq \inf\{t>0:W_t\notin D\}$ the first exit time from the set $D$. We define the killed process $W^D$ upon exiting the set $D$ by
\begin{align*}
	W^D_t\coloneqq\begin{cases}
		W_t,&t<\tau_D,\\
		\partial, &t\ge \tau_D,
	\end{cases}
\end{align*}
where $\partial$ is an additional point added to $\R^d$ called the cemetery.

Let $S$ be a subordinator independent of $W$, i.e. $S$ is an increasing L\'evy process such that $S_0=0$, with the Laplace exponent
\begin{align}\label{eq:Laplace exponent defn}
	\lambda\mapsto \phi(\lambda)=b\lambda+\int_{0}^\infty (1-e^{-\lambda\,t})\mu(dt),
\end{align}
where $b\ge0$ and the measure $\mu$ satisfies $\int_0^\infty (1\wedge t)\mu(dt)<\infty$. The measure $\mu$ is called the L\'evy measure and $b$ is the drift of the subordinator. 

The process $X=((X_t)_{t\ge0},(\p_x)_{x\in D})$ defined by $X_t\coloneqq W^D_{S_t}$ is called the subordinate killed Brownian motion. Here $\p_x$ denotes the probability under which the process $X$ starts from $x\in D$, and by $\ex_x$ we denote the corresponding expectation. As we will explain in Subsection \ref{ss:operator}, the operator $\Lo$ can be viewed as the infinitesimal generator of $X$.

An additional important process for studying equations related to $\Lo$ is the process $Y=((Y_t)_{t\ge0},(\p_x)_{x\in D})$ defined by $Y_t\coloneqq W_{S_t}$, called the subordinate Brownian motion. In other words, the process $W$ is not killed upon exiting the set $D$ before the subordination takes place. By using this close connection of the processes $X$ and 
 $Y$, later in the article we obtain the regularity of distributional solutions for $\Lo$.

\subsection{Assumptions}
Throughout the article, we always assume that $D$ is a bounded $C^{1,1}$ domain, although some auxiliary results are also valid for general open sets.

Before we introduce the assumption on the Laplace exponent $\phi$, i.e. on the subordinator $S$, we note that a function of the form \eqref{eq:Laplace exponent defn} is called a Bernstein function, and such functions characterize subordinators, see \cite[Chapter 5]{bernstein}.
The following assumption on $\phi$ is imposed throughout the article.
\begin{assumption}{WSC}{}\label{WSC}
	The function $\phi$ is a complete Bernstein function, i.e. the L\'evy measure $\mu(dt)$ has a completely monotone density $\mu(t)$, and  $\phi$ satisfies the following weak scaling condition at infinity: There exist $a_1,a_2>0$ and $\delta_1, \delta_2  \in(0,1)$ satisfying
	\begin{align}\label{eq:scaling condition}
		a_1\lambda^{\delta_1}\phi(t)\le \phi(\lambda t)\le a_2\lambda^{\delta_2}\phi(t),\quad t\ge1, \lambda\ge 1.
	\end{align}
\end{assumption}

The most notable subordinator with the property \ref{WSC} is the $\alpha$-stable subordinator where $\phi(\lambda)=\lambda^{\alpha/2}$, for some $\alpha\in(0,2)$. It satisfies the exact and even global scaling condition \eqref{eq:scaling condition}, and the process $Y$, in this case, is the isotropic $\alpha$-stable process. However, many other interesting subordinators fall within our setting, such as a sum of two stable subordinators, relativistic stable subordinators, tempered stable subordinators, and many more, see e.g. \cite[p. 409]{KSV_bhpinf} and \cite[p. 41]{vondra_heat}.

By the assumption \ref{WSC}, the function 
\begin{equation}\label{eq:phi_conjugate}
\phi^*(\lambda)\coloneqq \frac{\lambda}{\phi(\lambda)}
\end{equation}
is a complete Bernstein function, too, see \cite[Proposition 7.1]{bernstein}, and is called the conjugate Bernstein function of $\phi$. The relation \eqref{eq:scaling condition} also holds for $\phi^*$ but with different constants: i.e. it holds that
\begin{align}\label{eq:scaling condition phi*}
    {\frac1{a_2}} \lambda^{1-\delta_2}\phi^*(t)\le \phi^*(\lambda t)\le {\frac1{a_1}}\lambda^{1-\delta_1}\phi^*(t),\quad t\ge1, \lambda\ge 1.
\end{align}
Thus, \ref{WSC} also holds for $\phi^*$ and by $\nu(dt)=\nu(t)dt$ we denote the L\'evy measure of $\phi^*$.

In the next few paragraphs, we discuss properties of $\phi$. Since $\phi^*$ also satisfies the condition \ref{WSC}, the same will hold for the counterparts of $\phi^*$.
By a simple calculation, the scaling condition \eqref{eq:scaling condition} implies that $b=0$ in \eqref{eq:Laplace exponent defn} and the well-known bound
\begin{align}\label{eq:scaling and the derivative}
	\phi'(\lambda)\asymp \frac{\phi(\lambda)}{\lambda}, \quad\lambda\ge1.
\end{align}
Here, the lower bound follows from \eqref{eq:scaling condition} and the upper bound holds for every Bernstein function and every $\lambda>0$. The L\'evy measure $\mu(dt)$ is infinite, see \cite[p. 160]{bernstein}, and the density $\mu(t)$ cannot decrease too fast, i.e. there is $c=c(\phi)>1$ such that
$$ \mu(t)\le c\mu(t+1), \quad t\ge 1,$$
see \cite[Lemma 2.1]{ksv_twosided}. Furthermore, the density $\mu$ admits a sharp bound expressed via the function $\phi$, see \cite[Eq. (2.13)]{ksv_minimal2016} and \cite[Proposition 3.3]{mimica}, but we will not directly use this bound.

The potential measure $U$ of the subordinator $S$ is defined by $U(A)\coloneqq \int_0^\infty \p(S_t\in A)dt$, $A\in \BB(\R)$, and it has a decreasing density $\uu$ for which it holds that $\int_0^1 \uu(t)dt<\infty$, see \cite[Theorem 11.3]{bernstein}. The density $\uu$ also admits a sharp bound expressed via $\phi$, see \cite[Eq. (2.11)]{ksv_minimal2016} and \cite[Proposition 3.4]{mimica}, but in the article, they are used only in some results that we quote later. The potential density of $\phi^*$ will be denoted by $\vv$.

{Although} not all Bernstein functions satisfy \eqref{eq:scaling condition}, a general Bernstein function satisfies a version of a global scaling condition:
\begin{align}\label{eq:simple global scaling}
	1\wedge \lambda \le \frac{\phi(\lambda \, t)}{\phi(t)}\le 1\vee  \lambda,\quad \lambda>0,t>0,
\end{align}
which we get directly from \eqref{eq:Laplace exponent defn}.

Under the imposed assumptions of this article, important aspects of the potential theory of the process $X$ were developed in recent years in the articles \cite{song_vondra_JTP2006,ksv_minimal2016,ksv_potential_SKBM_2020,ksv_BdryKillLevy2020,Bio23}. These results include the scale invariant Harnack principle and the boundary Harnack principle, sharp estimates on the Green function and the jumping kernel, as well as the representation of harmonic functions. Although we do not use most of these properties directly, they are heavily used in the proofs of cited results and were essential for obtaining the results we build on.

	\subsection{Green functions}
The transition density of the Brownian motion $W$ we denote by
\begin{align}\label{eq:trans.dens.BM}
	p(t,x,y)=(4\pi t)^{-d/2}e^{-\frac{|x-y|^2}{4t}},\quad x,y\in\R^d,\, t>0.
\end{align}
It follows that the transition density of the killed Brownian motion $W^D$ is given, for all $x,y\in\R^d$, by
\begin{align}\label{eq:trans.dens. KBM}
	p_D(t,x,y)=p(t,x,y)-\ex_x[p(t-\tau_D,W_{\tau_D},y)\1_{\{\tau_D<t\}}]
\end{align}
It is well known that $p_D(t,\cdot,\cdot)$  is symmetric and  differentiable up to the boundary, i.e. $p_D(\cdot,\cdot,\cdot)\in C^1((0,\infty)\times \overline D\times \overline D)$, since $D$ is a $C^{1,1}$ open domain, see \cite[Lemma A.7]{Bio23}. In addition, we have the following heat kernel estimate: There exist constants $T_0=T_0(D)>0$, $c_1=c_1(T_0,D)>0$, $c_2=c_2(T_0,D)>0$, $c_3=c_3(D)>0$, and $c_4=c_4(D)>0$ such that for all $x,y\in D$ and $t\in(0,T_0]$ it holds that
\begin{align}
\begin{split}\label{eq:heat kernel estimate}
    &\left[\frac{\de(x)\de(y)}{t}\wedge 1\right]\frac{1}{c_1 	t^{d/2}}e^{-\frac{|x-y|^2}{c_2t}}\\&\hspace{6em}\le p_D(t,x,y)\le \left[\frac{\de(x)\de(y)}{t}\wedge 1\right]\frac{c_3}{t^{d/2}}e^{-\frac{c_4|x-y|^2}{t}}.
\end{split}
\end{align}
In fact, the right-hand side inequality in \eqref{eq:heat kernel estimate} holds for every $t>0$. For details, refer to \cite[Theorem 3.1 \& Theorem 3.8]{song_sharp_bounds_2004}, cf. \cite[Theorem 1.1]{zhang_heat_kernel} and \cite[Theorem 4.6.9]{davies_HeatKernel}.

The semigroup $(P_t)_{t\ge 0}$ of the process $W$ is given by
\begin{align}\label{eq:Brownian semigroup Rd}
	P_t f(x)=\int_{\R^d}p(t,x,y)f(y)dy=\ex_x[f(W_t)],
\end{align}
for  $f\in L^\infty(\R^d)\cup \BB_+(\R^d)$, which can be extended to a strongly continuous semigroup on $L^2(\R^d)$ and $C_0(\R^d)$, see \cite[Theorem 2.7]{chung_zhao}. It is very well known that $(P_t)_t$ satisfies both Feller, i.e. $P_t \big(C_0(\R^d)\big)\subseteq C_0(\R^d)$, and strong Feller property, i.e. $P_t \big(L^\infty(\R^d)\big)\subseteq C_b(\R^d)$, see  e.g. \cite[Section 1.1]{chung_zhao}. The potential kernel of $W$ (or the Green function of $W$) is given by
\begin{align*}
	G_{\R^d}(x,y)=\int_0^\infty p(t,x,y)dt=|x-y|^{-d+2},\quad x,y\in\R^d,
\end{align*}
since we consider only dimensions $d\ge 3$. The kernel $G_{\R^d}$ is the density of the mean occupation time for $W$, i.e. for $f\ge0$ we have
$$ \int_{\R^d} G_{\R^d}(x,y)f(y)dy=\ex_x\left[\int_0^\infty f(W_t)dt\right], \quad x\in D.$$

The semigroup $(P_t^D)_{t\ge 0}$ of the killed Brownian motion $W^D$ is for $f\in L^\infty(D)\cup \BB_+(\R^d)$ given by
\begin{align*}
	P_t^D f(x)=\int\limits_Dp_D(t,x,y)f(y)dy=\ex_x[f(W_t);t<\tau_D]=\ex_x[f(W^D_t)],
\end{align*}
where we set $f(\partial)=0$ for all Borel functions on $D$ by convention. Since $D$ is $C^{1,1}$, it is well known that the semigroup $(P_t^D)_{t\ge 0}$ is  strongly Feller and can be uniquely extended to a $L^2(D)$ and a $C_0(D)$ semigroup. For details see e.g. \cite[Chapter 2]{chung_zhao}.

The potential kernel of $W^D$ (or the Green function of $W^D$) is defined as
\begin{align*}
	\GD(x,y)=\int_0^\infty p_D(t,x,y)dt,\quad x,y\in\R^d.
\end{align*}
The kernel $\GD$ is symmetric, nonnegative, finite off the diagonal and jointly continuous in the extended sense, see \cite[Theorem 2.6]{chung_zhao}. Also, it is the density of the mean occupation time for $W^D$, i.e. for $f\ge0$ we have
$$ \int_D \GD(x,y)f(y)dy=\ex_x\left[\int_0^\infty f(W^D_t)dt\right]=\ex_x\left[\int_0^{\tau_D} f(W_t)dt\right], \quad x\in D.$$

The process $X$ is obtained by subordinating the killed Brownian motion $W^D$, hence the $L^2(D)$ transition semigroup of $X$, denoted by $(Q_t^D)_{t\ge0}$, is given by

\begin{align*}
	Q_t^D f=\int_0^\infty P_s^Df \,\p(S_t\in ds),\quad f\in L^2(D),
\end{align*}
see \cite[Proposition 13.1]{bernstein}, and admits the density 
\begin{align*}
	q_D(t,x,y)=\int_0^\infty p_D(s,x,y)\;\p(S_t\in ds).
\end{align*}
The semigroup $(Q_t^D)_t$ is also strongly Feller since $(P_t^D)_t$ is, see \cite[Proposition V.3.3]{bliedtner}.
The process $X$ also admits the potential kernel (i.e. the Green function of $X$) which is given by
\begin{align*}
	\GDfi(x,y)=\int_0^\infty q_D(t,x,y)dt=\int_0^\infty p_D(t,x,y)\uu(t)dt,\quad x,y\in\R^d.
\end{align*}
The kernel $\GDFI$ is symmetric, nonnegative, infinite on the diagonal, and finite and jointly continuous off the diagonal, see \cite[p. 857]{Bio23} and references therein. Moreover, $\GDFI$ is the density of the mean occupation time for $X$, i.e. for $f\ge0$ we have
$$ \int_D \GDFI(x,y)f(y)dy=\ex_x\left[\int_0^\infty f(X_t)dt\right], \quad x\in D.$$
In \cite[Lemma A.1]{Bio23}, building on \cite[Theorem 3.1]{ksv_minimal2016}, the sharp estimate for $\GDFI$ was obtained, i.e. we have for all $x,y\in D$
\begin{align}\label{eq:Green function sharp estimate}
	\GDfi(x,y)\asymp \left(\frac{\de(x)\de(y)}{|x-y|^{2}}\wedge 1\right)\frac{1}{|x-y|^{d}\phi(|x-y|^{-2})}.
\end{align}
This sharp estimate is not used directly in the article but was essential in some of the results that we cite.

The Poisson kernel of $X$ is defined by
\begin{align*}
		\PDfi(x,z)\coloneqq -\frac{\partial}{\partial {\mathbf{n}}}\GDfi(x,z)\asymp \frac{\de(x)}{|x-z|^{d+2}\phi(|x-z|^{-2})},\quad x\in D,z\in\partial D,
\end{align*}
where $\frac{\partial}{\partial \mathbf{n}}$ denotes the derivative in the direction of the outer normal. The well-definiteness as well as the sharp bound of the kernel $\PDFI$ were proved in \cite[Proposition 2.19]{Bio23}. The boundary condition in our semilinear problem contains the Poisson integral of the $d-1$ dimensional Hausdorff measure on $\partial D$, denoted by $\sigma$, which is defined as
\begin{align}\label{eq:PDFI_sigma}
     \PDFI\sigma(x)\coloneqq \int_{\partial D}\PDFI(x,z)\sigma(dz)\asymp \frac{1}{\de(x)^2\phi(\de(x)^{-2})},\enskip x\in D,
\end{align}
where the sharp bound for the Poisson integral above was proved in  \cite[Lemma 3.1]{Bio23}.

By the probabilistic characterization of Bernstein functions, the conjugate Bernstein function $\phi^*$ generates a subordinator $(T_t)_{t\ge0}$, see \cite[Chapter 5]{bernstein}. From what was already mentioned in the case of the function $\phi$, we also have in the case of the function $\phi^*$. In other words, the subordinator $(T_t)_{t\ge0}$ has a potential measure that also has the decreasing density which we denote by $V(dt)=\vv(t)dt$, see \cite[Theorem 11.3 \& Corollary 11.8]{bernstein}. We can also define the potential kernel generated by $\phi^*$ with
\begin{align}\label{eq:defn Green function 2}
	\GDfz(x,y)=\int_0^\infty p_D(t,x,y)\vv(t)dt,\quad x,y\in\R^d.
\end{align}
Recall that $\phi^*$ satisfies \ref{WSC}, too, so $\GDfz$ satisfies the same properties mentioned for $\GDFI$, e.g. it is symmetric, finite off the diagonal, jointly continuous, and satisfies the sharp bound \eqref{eq:Green function sharp estimate} where $\phi$ is replaced by $\phi^*$. Obviously, the kernel $\GDfz$ can be viewed as the potential kernel of the process $(W_{T_t}^D)_{t\ge 0}$.

As we have already mentioned, the process $Y$ is obtained by subordinating the Brownian motion $W$. Since $Y$ is a subordinate L\'evy process, it is again a L\'evy process, and by the construction, it has the characteristic exponent $\xi\mapsto \Psi(\xi)=\phi(|\xi|^2)$ of the form
\begin{align*}
		\Psi(\xi)=\phi(|\xi|^2)=\int_{\R^d}\left(1-\cos(\xi\cdot x)\right) J(dx),\quad \xi\in\R^d,
\end{align*}
where $J$, the so-called L\'evy measure of $Y$, satisfies $\int_{\R^d} (1\wedge |x|^2)J(dx)<\infty$. Also, $J$ has a density given by $J(x)=j(|x|)$, $x\in\R^d$, where
	\begin{align*}
		j(r)\coloneqq \int_0^\infty p(t,x,y)\mu(t)dt=\int_0^\infty (4\pi t)^{-d/2} e^{-r^2/(4t)}\mu(t)dt,\quad r>0.
	\end{align*}
The density $j$ is positive, continuous, decreasing and satisfies $\lim\limits_{r\to\infty}j(r)=0.$
It is well known that, if $\phi$ is a complete Bernstein function, then
\begin{align}
	j(r)\lesssim j(r+1)\enskip&\text{and }\enskip j(r)\lesssim j(2r), \quad r\ge 1,\label{eq:slow decr jump kern}\\
    \text{and}\hspace{2em} j(r)&\asymp \frac{\phi(r^{-2})}{r^d},\quad r\le1,\label{eq:sharp bnd jump kern}
\end{align}
where the constants of comparability depend only on $d$ and $\phi$, see e.g. \cite[Theorem 2.3, Eq. (2.11) \& Eq. (2.12)]{KSV_2012_SCM}. We also note that the upper bound for \eqref{eq:sharp bnd jump kern} holds for every $r>0$, see \cite[Eq. (15)]{bogdan_density_and_tails_unimodal}.

Since $d\ge 3$, the process $Y$ is transient, i.e. $\p_x(\lim_{t\to\infty}|X_t|=\infty)=1$, and we can define the potential kernel of $Y$, i.e. the Green function of $Y$, by
	\begin{align*}
		G^{\phi}_{\R^d}(x,y)\coloneqq G^{\phi}_{\R^d}(x-y)=\int_{0}^\infty p(t,x,y)\uu(t)dt, \quad x\in \R^d.
	\end{align*}
	The kernel $G^{\phi}_{\R^d}$ is the density of the mean occupation time for $Y$, i.e. for $f\ge0$ we have
	\begin{align*}
		\int_{\R^d}G^{\phi}_{\R^d}(x-y)f(y)dy=\mathbb{E}_x\left[\int_{0}^{\infty}f(Y_t)dt\right],\quad x\in\R^d.
	\end{align*}
	Further, $G^{\phi}_{\R^d}$ is jointly continuous, rotationally symmetric and radially decreasing, and from \cite[Lemma 3.2(b)]{KSV_global_UBHP} we get that for every $M>0$ there is a constant $C=C(d,\phi,M)>0$ such that 
	\begin{align*}
		C^{-1} \frac{1}{|x|^d\phi(|x|^{-2})}\le G^{\phi}_{\R^d}(x) \le	C \frac{1}{|x|^d\phi(|x|^{-2})},\quad |x|\le M.
	\end{align*}
 
Since $\phi^*$ also satisfies \ref{WSC}, when we subordinate the Brownian motion $W$ with the subordinator $T$, i.e. when the Laplace exponent of the subordinator is $\phi^*$, the relations mentioned above also hold where $\phi$ is replaced by $\phi^*$. For future reference, the Green function of this process, i.e. of $(W_{T_t})_t$, we denote by $G^{\phi^*}_{\R^d}$, and its jumping kernel by $j^*$.

To finish the subsection, we bring a connection of the kernels $\GD$, $\GDFI$ and $\GDfz$, and of the kernels $G_{\R^d}$, $G^{\phi}_{\R^d}$ and $G^{\phi^*}_{\R^d}$.
\begin{lem}\label{l: GDf(GDf*)=GD}
	Let $\Omega$ be $D$ or $\R^d$. For $x,y\in \Omega$ it holds that
	\begin{align}\label{eq: GDf(GDf*)=GD}
		\int_\Omega G^{\phi}_{\Omega}(x,\xi)G^{\phi^*}_{\Omega}(\xi,y)d\xi=G_{\Omega}(x,y).
	\end{align}
\end{lem}
\begin{proof}
	The claim follows from \cite[Proposition 14.2(ii)]{bernstein} where we set $\gamma=\delta_y$.
\end{proof}

\subsection{Operators \texorpdfstring{$\phi(-\Delta\vert_D)$}{phi(-Delta|D)} and \texorpdfstring{$\phi(-\Delta)$}{phi(-Delta)}}\label{ss:operator}
Let $\{\varphi_j\}_{j\in\N}$ 
be a Hilbert basis of $L^2(D)$ consisting of eigenfunctions 
of the Dirichlet Laplacian $-\Delta\vert_D$, 
associated to the eigenvalues $\lambda_j$, $j\in\N$, i.e.  $\varphi_j\in H^1_0(D)\cap C^{ \infty}(D)\cap C^{1,1}(\overline D)$ and
\begin{align}\label{eq:defn of eigenvalues}
	-\Delta\vert_D\varphi_j=\lambda_j\varphi_j, \quad \text{in $D$},
\end{align}
see \cite[Theorem 9.31]{brezis} and \cite[Section 8.11]{gilbarg_pde}. The relation \eqref{eq:defn of eigenvalues} can be viewed as a distributional or a pointwise relation. In addition, $\Delta\vert_D$ in \eqref{eq:defn of eigenvalues} can be viewed as the $L^2(D)$-infinitesimal generator of the semigroup $(P_t^D)_t$, i.e.
\begin{align*}
	\left.\Delta\right\vert_D u=\lim_{t\to 0}\frac{P_t^D u-u}{t}, \quad u\in \DD(\left.\Delta\right\vert_D).
\end{align*}
Here $\DD(\Delta\vert_D)$ denotes the domain of the generator $\Delta\vert_D$ and the above limit is taken with respect to $L^2(D)$ norm. Since $\varphi_j$ is an eigenfunction, we note that
\begin{align}\label{eq:Pt=e{-lt}}
    P_t^D\varphi_j=e^{-\lambda_j t}\varphi_j, \quad \text{in $D$}.
\end{align}
For the future reference, it is worth noting that the space $\DD(\left.\Delta\right\vert_D)$ is a class of functions $f\in H_0^1(D)$ such that $\Delta f$ exists in the weak distributional sense and belongs to $L^2(D)$, see \cite[Theorem 2.13]{chung_zhao}. In particular, $C_c^2(D)\subset \DD(\left.\Delta\right\vert_D)$. For more details on the operator $\left.\Delta\right\vert_D$, refer to \cite[Chapter 2]{chung_zhao} and \cite[Chapter 9]{brezis}. 

By the assumption that $D$ is $C^{1,1}$, it follows that $0<\lambda_1< \lambda_2\le\lambda_3\le \dots$, and the Weyl's asymptotic law holds:
\begin{align}\label{eq:Weyl's law}
	\lambda_j \asymp j^{2/d},\quad j\in \N.
\end{align}
We may also choose the basis $\{\varphi_j\}_{j\in \N}$ such that $\varphi_1>0$ in $D$, see \cite[Chapter 9]{brezis}, so that very important sharp estimate for $\varphi_1$ holds:
\begin{align}\label{eq:varphi=de sharp estimate}
	\varphi_1(x)\asymp \de(x),\quad x\in D.
\end{align}
The interior estimate is trivial and the boundary estimate follows from Hopf's lemma. We note that this estimate is closely connected to the boundary behaviour of the heat kernel $p_D$, as one can see in \eqref{eq:heat kernel estimate}, and in the choice of the domain for the distributional operators in $D$, as we will see in a moment.

Consider the Hilbert space
	\begin{align*}
		H_D(\phi)\coloneqq \left\{v=\sum_{j=1}^\infty \widehat v_j\varphi_j\in L^2(D):\|v\|_{H_D(\phi)}^2\coloneqq\sum_{j=0}^\infty\phi(\lambda_j)^{2}\vert\widehat v_j\vert^2<\infty\right\}.
	\end{align*}
	The spectral operator $\Lo:H_D(\phi)\to L^2(D)$ is defined as
	\begin{align}\label{eq:definition of phi(Delta D)}
		\Lo u=\sum_{j=1}^\infty\phi(\lambda_j)\widehat u_j\varphi_j, \quad u\in H_D(\phi).
	\end{align}
 In the simple case where $\phi(t)=t^s$, the operator $\Lo=(-\Delta_{|D})^s$ is called the spectral fractional Laplacian.
 It was shown in \cite[Section 2]{Bio23} that the operator $\Lo$ is the $L^2(D)$-infinitesimal generator of the process $X$. Furthermore, it was also shown that $\Lo$ has a pointwise representation: For $u\in C^{1,1}(D)\cap \DD(\left.\Delta\right|_D)$ and for a.e. $x\in D$ we have
		\begin{align}\label{eq:Lo pointwisely}
			\Lo u(x)=\textrm{P.V.}\int\limits_D [u(x)-u(y)]J_D(x,y)dy + \kappa(x)u(x),
		\end{align}
		where
		\begin{align}\begin{split}\label{e:J_D}
		    J_D(x,y)&\coloneqq \int_0^\infty p_D(t,x,y)\mu(t)dt,\\\kappa(x)&\coloneqq\int_0^\infty \left(1-\int_Dp_D(t,x,y)dy\right)\mu(t)dt.
		\end{split}
		\end{align}
For the jumping kernel $J_D(x,y)$, the following sharp estimate holds
\begin{equation}\label{eq:J_D estimate}
    J_D(x,y)\asymp\left(\frac{\de(x)\de(y)}{|x-y|^2}\wedge 1\right)j(|x-y|),\quad x,y\in D,
\end{equation}
see \cite[Proposition 3.5]{ksv_minimal2016} and \cite[Remark 2.5]{Bio23}.

We also define a distributional version of $\Lo$ for all $u\in \LLL$, which we denote by $\wLo u$, by the following relation
\begin{align*}
    \langle \wLo u,\psi\rangle\coloneqq \langle  u,\Lo\psi\rangle\coloneqq\int_D u(x)\Lo \psi (x)dx,
\end{align*}
for all $\psi\in C_c^\infty(D)$. The space $\LLL$ is the appropriate one for the distributional operator because $|\Lo\psi|\le \de$ in $D$, and $-\Lo\psi\asymp \de$ near $\partial D$ for $\psi\ge0$, as shown in \cite[Lemma 2.11]{Bio23}. For additional details on $\Lo$ and $\wLo$, please refer to \cite[Section 2]{Bio23} and the references therein.

\begin{rem}\label{r:pointwise solution}
    By analyzing the proofs of \cite[Lemma 2.4]{Bio23} and \cite[Remark 2.6]{Bio23}, one can easily see that the expression \eqref{eq:Lo pointwisely} makes sense for all $C^{2\delta_2+\varepsilon}(D)\cap \LLL$ functions, where $\varepsilon>0$. This is because \ref{WSC} implies
    \[
    J_D(x,y)\lesssim j(|x-y|)\lesssim |x-y|^{-d-2\delta_2},\quad x,y\in D.
    \]
    Thus, we extend the definition of $\Lo u$ by \eqref{eq:Lo pointwisely} whenever the principal value integral in \eqref{eq:Lo pointwisely} exists.
   Furthermore, if $u\in C^{2\delta_2+\varepsilon}(D)\cap \LLL$ then $\wLo u=\Lo u$ as distributions in $D$.
\end{rem}

Now we turn to the definition of the operator $\LoR$. We define the operator $-\LoR$ as the infinitesimal generator of the $C_0(\R^d)$-semigroup generated by the L\'evy process $Y$. It is well known that the space $C_c^{\infty}(\R^d)$ is contained in the domain of this generator, and for $u\in C_c^{\infty}(\R^d)$ it holds that
\begin{align}
-\LoR u(x)&=\int_{\R^d}\left(u(y)-u(x)-  \nabla u(x) \cdot (y-x)\1_{\{|y-x|\le 1\}}\right)j(|y-x|)\, dy \nonumber\\
&=\textrm{P.V.}\int_{\R^d} \left(u(y)-u(x)\right)j(|y-x|)dy\label{eq:LoR defn}\\
&=\lim_{\epsilon \to 0} \int_{|y-x|>\epsilon} \left(u(y)-u(x)\right)j(|y-x|)\, dy\nonumber\\
&=\lim_{\epsilon \to 0} \int_{|h|>\epsilon} \left(u(x+h)-u(x)\right)j(|h|)\, dh.\nonumber
\end{align}
In the familiar case of the isotropic stable process, where $\phi(t)=t^s$, the operator $-\LoR$ is the fractional Laplacian $-(-\Delta)^s$.

We extend the definition of $\LoR u(x)$ by \eqref{eq:LoR defn} whenever the limit in \eqref{eq:LoR defn} exists. This is true e.g. for all $x\in \R^d$ when $u\in C^{2\delta_2+\varepsilon}(\R^d)\cap L^1(\R^d,(1\wedge j(|x|))dx)$, where $\varepsilon>0$.
For functions $u\in  L^1(\R^d,(1\wedge j(|x|))dx)$ we also define a distributional version of $\LoR$, denoted by $\wLoR$, by the relation
	\begin{align*}
		\langle \wLoR u,\varphi\rangle\coloneqq \langle u,\LoR\varphi\rangle\coloneqq \int_{\R^d}u(x)\LoR\varphi(x)dx,\quad \varphi\in C_c^\infty(\R^d).
	\end{align*}
	It is easy to show that if $u\in C^{2\delta_2+\varepsilon}(\R^d)\cap L^1(\R^d,(1\wedge j(|x|))dx)$ then $\wLoR u=\LoR u$ as distributions in $\R^d$. However, when dealing with $\wLoR u$ we sometimes just simply say that we look at $\LoR u$ in the distributional sense.

\subsection{Different types of solutions and boundary conditions for the Dirichlet problem}\label{ss:dirichlet problem}
In the article, we deal with three types of solutions to the Dirichlet problem involving the operator $\Lo$, and we also have two types of the boundary condition.

We say that $u:D\to \R$ is a pointwise solution to 
\begin{equation}\label{eq:generic problem}
	\begin{array}{rcll}
		\Lo u&=& f& \quad \text{in } D,\\
		\frac{u}{\PDFI\sigma}&=&g&\quad \text{on }\partial D,
	\end{array}
\end{equation}
if $\Lo u =f$ in $D$ in the sense of \eqref{eq:Lo pointwisely} and if
\begin{align}\label{eq:boundary Dirichlet limit}
    \lim_{D \ni x \to z}\frac{u(x)}{\PDFI\sigma(x)}=g(z),\quad  z\in \partial D.
\end{align}

On the other hand, the function $u:D\to \R$ is a distributional solution to \eqref{eq:generic problem} if $\Lo u =f$ in $D$ in the distributional sense. Additionally, for the boundary condition, we can either look at the pointwise limit as in \eqref{eq:boundary Dirichlet limit} or we can look at the so-called weak $L^1$ boundary condition: 
\begin{align}\label{eq:distri solution boundary}
			\lim_{t\downarrow 0}\frac1t \int_{\{\de(x)\le t\}}\frac{u(x)}{\PDFI\sigma(x)}\varphi(x)dx=\int_{\partial D}\varphi(z)g(z)\sigma(dz),
\end{align}
for all $\varphi \in C(\overline D)$. It is easy to see that if the boundary limit \eqref{eq:boundary Dirichlet limit} exists and it is finite (or infinite everywhere), then the weak $L^1$ boundary condition \eqref{eq:distri solution boundary} also holds with the same limiting function.
However, in Section \ref{s:regularity} the boundary condition in \eqref{eq:generic problem} is left out which means that we only demand that $\Lo u=f$ holds in the distributional sense.

Finally, we say that $u\in L^1_{loc}(D)$ is a weak solution to \eqref{eq:generic problem} if for every $\xi\in C_c^\infty(D)$ it holds that
		\begin{align}\label{eq:problem - integral definition}
			\int_D u(x)\xi(x)dx=\int_D\GDfi\xi(x)f(x)dx-\int_{\partial D}\frac{\partial}{\partial \mathbf{n}}\GDfi\xi(z)g(z)\sigma(dz).
		\end{align}
Equivalently, by writing $\psi=\GDFI\xi$, the relation \eqref{eq:problem - integral definition} says
\begin{align}\label{eq:problem - integral definition1}
			\int_D u(x)\Lo \psi(x)dx=\int_D\psi(x) f(x)dx-\int_{\partial D}\frac{\partial}{\partial \mathbf{n}}\psi(z)g(z)\sigma(dz).
\end{align}
We refer to \cite{Bio23} for a detailed study of weak solutions and their relation to pointwise and distributional solutions.

The following lemma explains why a solution to \eqref{eq:problem_intro}, the so-called large solution, is the one that is not uniformly bounded by a nonnegative harmonic function.
\begin{lem}\label{l:large solution harmonic unbounded}    
    If $u:D\to \R$ satisfies 
    \begin{align}\label{eq:large solution harmonic0}
        \lim_{D\ni x\to z}\frac{|u(x)|}{\PDFI\sigma(x)}=\infty,\quad z\in\partial D,
    \end{align}
    then $u$ is not uniformly bounded in $D$ by any nonnegative harmonic function with respect to $\Lo$.
\end{lem}

\begin{proof}
    Since $\partial D$ is compact, by the standard $3\varepsilon$ argument, we easily obtain uniformness of the convergence in \eqref{eq:large solution harmonic0}.
    
    Suppose that the claim of the lemma does not hold, i.e. there exists $h$, a nonnegative harmonic function with respect to $\Lo$, such that $|u(x)|\le h(x)$, $x\in D$. By the representation of nonnegative harmonic functions, see \cite[Theorem 2.23]{Bio23}, there exists a finite measure $\zeta$ on $\partial D$ such that $h(x)=\PDFI\zeta(x)$, $x\in D$. Further, we have
    \begin{align}\label{eq:large solution harmonic1}
        \lim_{D\ni x\to z}\frac{\PDFI\zeta(x)}{\PDFI\sigma(x)}=\lim_{D\ni x\to z}\frac{\PDFI\zeta(x)}{|u(x)|}\frac{|u(x)|}{\PDFI\sigma(x)}=\infty,\quad z\in \partial D,
    \end{align}
    and the convergence is uniform by the same argument as in the beginning of the proof.

    However, by the proof of \cite[Proposition 3.5]{Bio23}, we get that
    \begin{align*}
        \limsup_{t\to 0}\frac{1}{t}\int_{\{\de(x)<t\}}\frac{\PDFI\zeta(x)}{\PDFI\sigma(x)}dx&=\limsup_{t\to 0}\int_{\partial D}\left(\frac{1}{t}\int_{\{\de(x)<t\}}\frac{\PDFI(x,z)}{\PDFI\sigma(x)}dx\right)\zeta(dz)\\
        &\le c\int_{\partial D} \zeta(dz)=c\,\zeta(\partial D)<\infty,
    \end{align*}
    which is a contradiction with \eqref{eq:large solution harmonic1} which says that $\frac{\PDFI\zeta(x)}{\PDFI\sigma(x)}$ is arbitrarily and uniformly large near the boundary.
\end{proof}

To conclude this subsection, we would like to make a few remarks. When dealing with distributional and weak solutions, it is possible to consider Radon measures $\mu$ and $\zeta$ instead of the functions $f$ and $g$ as the boundary data, provided that they have the necessary integrability properties. Further, every weak solution is also a distributional one with the weak $L^1$ boundary condition, as proved in \cite[Proposition 4.6]{Bio23}. However, it is important to keep in mind that a weak solution may not have a pointwise boundary limit in the sense of \eqref{eq:distri solution boundary}. On the other hand, if a distributional or weak solution $u$ satisfies \eqref{eq:boundary Dirichlet limit} and has sufficient regularity, it is also a pointwise solution, see Remark \ref{r:pointwise solution}.

\subsection{Renewal function}
We end the section with a short introduction of the renewal function $V$. For details, see \cite[Chapter VI]{Ber}. The renewal function $V$ plays a prominent role in studying boundary behaviour in $C^{1,1}$ domains. We refer the interested reader to \cite{ksv_twosided,KMR_suprema,semilinear_cvw,semilinear_bvw}.

Let $Z=(Z_t)_{t\ge 0}$ be a one-dimensional subordinate Brownian motion with the characteristic exponent $\phi(\theta^2)$, $\theta\in \R$. E.g. $Z$ can be thought of as one of the components of the process $Y$. Let $M_t:=\sup_{0\le s\le t} Z_s$ be the supremum process of $Z$, and let $L=(L_t)_{t\ge 0}$ be the local time of $M_t-Z_t$ at zero. The inverse local time $L_t^{-1}:=\inf\{s>0: L_s>t\}$ is called the ascending ladder time process of $Z$. Define the ascending ladder height process $H=(H_t)_{t\ge 0}$ of $Z$ by $H_t:=M_{L_t^{-1}}=Z_{L_t^{-1}}$ if $L_t^{-1}<\infty$, and $H_t=\infty$ otherwise. We define the renewal function $V$ of the process $H$ by
\begin{equation}\label{eq:V}
    V(t):=\int_0^{\infty}\p(H_s\le t)\, ds, \quad t\in \R.
\end{equation}
It holds that  $V$ is strictly increasing, $V(t)=0$ for $t<0$, $V(0)=0$, and $V(\infty)=\infty$. The most notable property of the renewal function $V$ is that $V_{|(0,\infty)}$ is harmonic with respect to the killed process $Z^{(0,\infty)}$, which was first used in \cite{ksv_twosided} to obtain the precise decay rate of harmonic functions of a $d$-dimensional subordinate Brownian motion. 

In the case of the isotropic $\alpha$-stable process, we have $V(t)=t^{\alpha/2}$, but in general, the function $V$ is not known explicitly. However, under \ref{WSC} it follows from \cite[Lemma 2.4, Lemma 2.5]{KKLL}, cf. \cite{ksv_twosided}, that $V\in C^2$ and that for every $R>0$:
\begin{align}
V(t) &\asymp \Phi(t):=\phi(t^{ -2})^{-1/2}\, ,\quad 0<t\leq R,\label{e:V-phi}\\
V'(t)&\lesssim\frac{V(t)}{t}\, ,\quad 0<t\leq R,\label{eq:H1}\\
|V''(t)|&\lesssim\frac{V'(t)}{t}\, ,\quad 0<t\leq R.\label{eq:H2}
\end{align}

\section{Regularity of distributional solutions}\label{s:regularity}
In this section, we prove two higher regularity theorems for distributional solutions to linear non-local problems associated with operators $\LoR$ and $\Lo$. Analogs of these results were already established in the fractional setting in \cite{Silvestre2007regularity} and \cite{AbaDupaNonhomo2017}. Our generalisations of these results follow the approach of the fractional setting, with necessary modifications due to the lack of exact scaling. The results of this section are self-contained and fill in some minor gaps in the original proofs. As we already pointed out in the introduction, the choice to work with classical H\"older spaces, instead of generalised H\"older spaces corresponding to $\phi$ as in \cite{BaeKassmann} and \cite{KimLee}, was made in order to obtain regularity results for all ranges of Matuszewska indices in \eqref{eq:simple global scaling}. These results are essential in the construction of the large solution by an approximating sequence in Lemma \ref{l:large finite}.  

To simplify notation and to make our results more readable, we write $C^{k+\alpha}$ instead of $C^{k,\alpha}$ for $k\in \N_0$ and $\alpha\in(0,1]$. Or, more generally, for $\beta\in (0,\infty)$, $C^\beta$ means $C^{\lfloor \beta \rfloor, \beta -\lfloor \beta \rfloor}$. However, in cases where more clarity is needed, we stick with the notation of $C^{k,\alpha}$-type.

Further, when dealing with distributional solutions which are defined almost everywhere, we say that a solution has a certain regularity if there exists a modification (with respect to the Lebesgue measure) of this function with the stated regularity.

\begin{thm}\label{t:regularity SBM}
	Let $f\in C^{k+\alpha}(\R^d)$ for $k\in \N_0$ and $\alpha\in (0,1)$ so that $k+\alpha+2\delta_1\notin \N$, and let $u$ be a distributional solution to
	\begin{align}\label{eq:distr.linear solution}
		\LoR u=f,\quad \text{in $\R^d$}.
	\end{align} 
Then $u\in C^{k+\alpha+2\delta_1}(\R^d)$ and there exists a constant $C=C(d,k,\alpha,\phi)>0$ such that
    \begin{align}\label{eq:regularity SBM no1}
		||u||_{C^{k+\alpha+2\delta_1}(\R^d)}&\le C\,||f||_{C^{k+\alpha}(\R^d)}.
	\end{align}
	On the other hand, if $f\in L^\infty(\R^d)$, then $u\in C^{\beta}(\R^d)$ for all $\beta\in(0,2\delta_1)$ and there exists a constant $C=C(d,\phi)>0$ such that
	\begin{align}\label{eq:regularity SBM no2}
		||u||_{C^{\beta}(\R^d)}&\le C\,||f||_{L^\infty(\R^d)}.
	\end{align}
\end{thm}
In essence, the previous theorem says that, given some regularity on $f$, the H\"older regularity of a solution to \eqref{eq:distr.linear solution} increases by (at least) $2\delta_1$. In the fractional Laplacian case (where $\delta_1=s$), this result has been proven in \cite[Section 2.1]{Silvestre2007regularity}. 

In order to prove Theorem \ref{t:regularity SBM}, we show a few auxiliary results. The first lemma is a generalisation of \cite[Proposition 2.5]{Silvestre2007regularity}.

\begin{lem}\label{l:reg. 0,alpha}
	Let $\alpha\in(0,1]$ and suppose $\alpha>2\delta_2$. If $u\in C^{k,\alpha}(\R^d)$, for $k\in \N_0$, then $\LoR u\in C^{k,\alpha-2\delta_2}(\R^d)$, and it holds that
    \begin{align}\label{eq:reg. 0,alpha}
        \|\LoR u\|_{C^{k,\alpha-2\delta_2}(\R^d)}\le c(d,k,\alpha,\phi)\|u\|_{C^{k,\alpha}(\R^d)}. 
    \end{align}
\end{lem}

\begin{proof}
	Take $k=0$. First, we prove that $\LoR u$ is well defined. Note that, by using \eqref{eq:sharp bnd jump kern} and \ref{WSC}, we have 
    \begin{align}\label{eq:jump kernel scaled}
        j(r)\lesssim r^{-d-2\delta_2}, \quad r\le1,    
    \end{align}
    and the constant of comparability depends only on $d$ and $\phi$.
    Since $u\in C^{0,\alpha}(\R^d)$ and  $\alpha>2\delta_2$ we have for every $x\in\R^d$
    \begin{align*}
        &\left|\int_{\R^d}(u(x+h)-u(x))j(|h|)dh \right|\\
        \qquad&\le \int_{\R^d)}|u(x+h)-u(x)|\left(1_{B(0,1)}+1_{B(0,1)^c}\right)j(|h|)dh\\
        \qquad&\lesssim \int_{B(0,1)}[u]_{0,\alpha,\R^d}|h|^{\alpha}\frac{1}{|h|^{d+2\delta_2}}dh+2\|u\|_{L^\infty(\R^d)}\int_{B(0,1)^c}j(|h|)dh\\
        \qquad&\le c_1 \|u\|_{C^{0,\alpha}(\R^d)}<\infty,
    \end{align*}
    for some $c_1=c_1(d,\alpha,\phi)>0$, i.e. $\LoR u$ is well defined and 
    \begin{align}\label{eq:LoR bounded}
        \|\LoR u\|_{L^\infty(\R^d)}\le  c_1\|u\|_{C^{0,\alpha}(\R^d)}.    
    \end{align}
To show the desired regularity of $\LoR$ take $x_1,x_2\in\R^d$ and $r<1$ and note that
    \begin{align*}
        &\hspace{-1em}|\LoR u(x_1)-\LoR u(x_2)|
        \\&\le \int_{B(0,r)}(|u(x_1+h)-u(x_1)|+|u(x_2+h)-u(x_2)|)j(|h|)dh\\
        &+\int_{B(0,r)^c}(|u(x_1+h)-u(x_2+h)|+|u(x_2)-u(x_1)|)j(|h|)dh\\
        &\eqqcolon I_1+I_2.
    \end{align*}
    Since $|u(x)-u(y)|\le [u]_{0,\alpha,\R^d}|x-y|^\alpha$, $x,y\in\R^d$, by applying \eqref{eq:jump kernel scaled} we get that
    \begin{align}\label{eq:regularity I1 bound}
        I_1\lesssim \int_{B(0,r)}[u]_{0,\alpha,\R^d}|h|^{\alpha}\frac{1}{|h|^{d+2\delta_2}}dh\le c_2 [u]_{0,\alpha,\R^d}r^{\alpha-2\delta_2},
    \end{align}
    for some $c_2=c_2(d,\alpha,\phi)>0$. Similarly, for some  $c_3=c_3(d,\alpha,\phi)>0$ we have
    \begin{align}
        I_2&\le2 \int_{B(0,r)^c}[u]_{0,\alpha,\R^d}|x_1-x_2|^\alpha j(|h|)dh\nonumber\\
        &=2\int_{B(0,r)^c} [u]_{0,\alpha,\R^d}|x_1-x_2|^\alpha (1_{B(0,1)\setminus B(0,r)}+1_{B(0,1)^c})j(|h|)dh\nonumber\\
        &\lesssim  [u]_{0,\alpha,\R^d}|x_1-x_2|^\alpha \int_{B(0,1)\setminus B(0,r)}|h|^{-d-2\delta_2}dh\nonumber\\&
        \hspace{6em}
        + [u]_{0,\alpha,\R^d}|x_1-x_2|^\alpha\int_{B(0,1)^c}j(|h|)dh\nonumber\\
        &\le c_3[u]_{0,\alpha,\R^d}\big(|x_1-x_2|^\alpha r^{-2\delta_2}+|x_1-x_2|^\alpha\big).\label{eq:regularity I2 bound}
    \end{align}

    If $|x_1-x_2|\le 1$, then $|x_1-x_2|^\alpha\le |x_1-x_2|^{\alpha-2\delta_2}$, so by choosing $r=|x_1-x_2|$ in \eqref{eq:regularity I1 bound} and \eqref{eq:regularity I2 bound} we obtain
    \begin{align*}
        \frac{|\LoR u(x_1)-\LoR u(x_2)|}{|x_1-x_2|^{\alpha-2\delta_2}}\le (c_2+c_3) [u]_{0,\alpha,\R^d}.
    \end{align*}
   If $|x_1-x_2|> 1$, then by \eqref{eq:LoR bounded} 
    \begin{align*}
        \frac{|\LoR u(x_1)-\LoR u(x_2)|}{|x_1-x_2|^{\alpha-2\delta_2}}\le  2c_1\|u\|_{C^\alpha(\R^d)},
    \end{align*}
 which, together with the previous inequality, implies \eqref{eq:reg. 0,alpha} for $k=0$.

For $k\in \N$ first note that the derivative and the operator $\LoR$ commute, i.e.
    \begin{align*}
        \partial_{x_i}\LoR u(x)&=\lim_{t\to 0}\int_{\R^d}\frac 1 t (u(x+te_i+h)-u(x+te_i)-u(x+h)+u(x))j(|h|)dh\\
        &= \LoR (\partial_{x_i}u)(x).
    \end{align*}
 Indeed, similarly as before
we note that $|\nabla u(x)-\nabla u(y)|\lesssim |x-y|^\alpha$, $x,y\in\R^d$, the result follows by the dominated convergence theorem. Now the desired result for $k\in \N$ follows by the iteration of this argument and the result for $k=0$. 
\end{proof}

Together with the previous lemma, the following result generalises \cite[Proposition 2.6]{Silvestre2007regularity}. However, here we slightly deviate from the ideas in the proof in \cite{Silvestre2007regularity}, as we apply a more direct approach without using the Riesz transform.

\begin{lem}\label{l:reg. 1,alpha}
	Let $\alpha\in(0,1]$ and suppose $1+\alpha>2\delta_2>\alpha$. If $u\in C^{k,\alpha}(\R^d)$, for $k\in \N$, then $\LoR u\in C^{k-1,1+\alpha-2\delta_2}$, and it holds that
    \begin{align}\label{eq:reg. 1,alpha}
        \|\LoR u\|_{C^{k-1,1+\alpha-2\delta_2}(\R^d)}\le c(d,k,\alpha,\phi)\|u\|_{C^{k,\alpha}(\R^d)}. 
    \end{align}
\end{lem}

\begin{proof}
	We follow the approach from the previous lemma. First, take $k=1$ and note that
 \begin{align}\label{eq:LoR bounded no2}
     \| \LoR u\|_{L^\infty(\R^d)}\le c_1(d,\alpha,\phi) \| u\|_{C^{1,\alpha}(\R^d)}.
 \end{align}
 Indeed, since $j$ is symmetric, by applying \eqref{eq:jump kernel scaled} we get that
 \begin{align*}
        |\LoR u(x)|
        &\le \int_{\R^d}|u(x+h)-u(x)-\nabla u(x)\cdot h \,\1_{B(0,1)}(h)|j(h)dh \\ 
        &\lesssim \int_{B(0,1)}[u]_{1,\alpha,\R^d}|h|^{1+\alpha}\frac{1}{|h|^{d+2\delta_2}}dh+\|u\|_{L^\infty(\R^d)}\int_{B(0,1)^c}j(|h|)dh\\
        &\le c_1(d,\alpha,\phi)\,\|u\|_{C^{1,\alpha}(\R^d)}<\infty.
\end{align*}
Next, for $x_1,x_2\in\R^d$ we estimate the difference
 \begin{align*}
        &|\LoR u(x_1)-\LoR u(x_2)|\le I_1+I_2+I_3,
    \end{align*}
\[
        I_i=\left|\int_{B_i}\left(\sum_{j=1}^2 (-1)^j {\big(}u(x_j+h)-u(x_j)-\nabla u(x_j)\cdot h\,\1_{B(0,1)}(h){\big)}\right)j(|h|)dh\right|,
\]
where $r<1$ and $B_1=B(0,r)$, $B_2=B(0,1)\setminus B(0,r)$, $B_3=B(0,1)^c$. By applying the estimate $|u(x+h)-u(x)-\nabla u(x)\cdot h|\le [u]_{1,\alpha,\R^d}|h|^{1+\alpha}$ and  \eqref{eq:jump kernel scaled}, we have
\begin{align}\label{eq:reg. 1,alpha no1}
    I_1\lesssim [u]_{1,\alpha,\R^d}\int_0^r h^{\alpha-2\delta_2}dh=c_2(d,\alpha,\phi)[u]_{1,\alpha,\R^d}r^{1+\alpha-2\delta_2}.
\end{align}
For $I_2$ we have, by Taylor's expansion,
\begin{align}\label{eq:Taylors expansion Holder 1+alpha}
    u(x)=u(y)+\nabla u(y)\cdot({x-y})+H^{1,\alpha}_{x,y}({x-y}),
\end{align}
where the remainder satisfies $|H^{1,\alpha}_{x,y}({x-y})|\le [u]_{1,\alpha,\R^d}|{x-y}|^{1+\alpha}$. Hence, by using \eqref{eq:Taylors expansion Holder 1+alpha}, \eqref{eq:jump kernel scaled} and symmetry of $j$ in $B_2$ (which implies that the third term is equal to zero) we have
\begin{align}
    I_2&\le\int_{B_2}\big|u(x_1+h)-u(x_2+h)-\nabla u(x_2+h)\cdot ({x_1-x_2})|j(|h|)dh\nonumber\\
    &+|u(x_1)-u(x_2)-\nabla u(x_2)\cdot ({x_1-x_2})|\int_{B_2}j(|h|)dh\nonumber\\
    &+\left|\int_{B_2}(\nabla u(x_1)-\nabla u(x_2))\cdot hj(|h|)dh\right|\nonumber\\
    &+\big|\int_{B_2}(\nabla u(x_2+h)-\nabla u(x_2))\cdot ({x_1-x_2})j(|h|)dh\big|\nonumber\\
    &\lesssim [u]_{1,\alpha,\R^d}|{x_1-x_2}|^{1+\alpha}\int_r^1\frac{dh}{h^{1+2\delta_2}}+[u]_{1,\alpha,\R^d}|{x_1-x_2}|\int_r^1\frac{dh}{h^{1+2\delta_2-\alpha}}\nonumber\\
    &\le c_3(d,\alpha,\phi)[u]_{1,\alpha,\R^d}\big(|{x_1-x_2}|^{1+\alpha}r^{-2\delta_2}+|{x_1-x_2}|r^{\alpha-2\delta_2}\big).\label{eq:reg. 1,alpha no2}
\end{align}
Analogously, since $|u(x_1)-u(x_2)|\vee |u(x_1+h)-u(x_2+h)|\le \|u\|_{C^{1,\alpha}(\R^d)}|x_1-x_2|$ we have
\begin{align}
\begin{split}\label{eq:reg. 1,alpha no3}
    I_3&\le 2\|u\|_{C^{1,\alpha}(\R^d)}|x_1-x_2|\int_{B(0,1)^c}j(|h|)dh\\&\le c_4(d,\alpha,\phi)\|u\|_{C^{1,\alpha}(\R^d)}|x_1-x_2|.
\end{split}
\end{align}

If $|x_1-x_2|\le 1$, take $r=|x_1-x_2|$ in \eqref{eq:reg. 1,alpha no1} and \eqref{eq:reg. 1,alpha no2}, and in \eqref{eq:reg. 1,alpha no3} note that $1+\alpha-2\delta_2\in(0,1)$ to obtain
\begin{align}\label{eq:reg. 1,alpha no4}
    |\LoR u(x_1)-\LoR u(x_2)|\le c_5(d,\alpha,\phi)\|u\|_{C^{1,\alpha}(\R^d)}|x_1-x_2|^{1+\alpha-2\delta_2}.
\end{align}
If $|x_1-x_2|\ge 1$, then by \eqref{eq:LoR bounded no2} we trivially have
\begin{align}\label{eq:reg. 1,alpha no5}
|\LoR u(x_1)-\LoR u(x_2)|\le  2c_1\|u\|_{C^{1,\alpha}(\R^d)}|x_1-x_2|^{1+\alpha-2\delta_2}.
    \end{align}
Hence, the case $k=1$ follows by \eqref{eq:LoR bounded no2}, \eqref{eq:reg. 1,alpha no4}, and \eqref{eq:reg. 1,alpha no5}. For $k\ge 2$, in the same way as in the previous lemma, we note that the derivative and the operator $\LoR$ commute.
\end{proof}

The last auxiliary result, which seems to be missing in \cite{Silvestre2007regularity}, covers the third case of the range of exponents, i.e. the case when $\delta_2$ is close to 1.
\begin{lem}\label{l:reg. 2,alpha}If $u\in C^{k,\alpha}(\R^d)$, for $k\in \N$ and $k\ge2$, and $1+\alpha<2\delta_2$, then $\LoR u\in C^{k-2,2+\alpha-2\delta_2}(\R^d)$, and it holds that
    \begin{align}\label{eq:reg. 2,alpha}
        \|\LoR u\|_{C^{k-2,2+\alpha-2\delta_2}(\R^d)}\le c(d,k,\alpha,\phi)\|u\|_{C^{k,\alpha}(\R^d)}. 
    \end{align}
\end{lem}

\begin{proof}
Similarly as in the previous lemma, for $u\in C^{2,\alpha}(\R^d)$ by Taylor's expansion we have
    \begin{align*}
        u(x+h)=u(x)+\nabla u(x)\cdot h+\frac12 Hu(x)(h,h)+H^{2,\alpha}_{x,x+h}(h), 
    \end{align*}
    where the remainder satisfies $|H^{2,\alpha}_{x,x+h}(h)|\le [u]_{2,\alpha,\R^d}|h|^{2+\alpha}$. Hence, 
    \begin{align}\label{eq:LoR bounded no3}
        \| \LoR u\|_{L^\infty(\R^d)}\le c_1(d,\alpha,\phi) \| u\|_{C^{2,\alpha}(\R^d)}.
     \end{align}
     To estimate the difference $|\LoR u(x_1)-\LoR u(x_2)|$ for $x_1,x_2\in\R^d$, we break $\LoR$ into the three integrals $I_1$, $I_2$, and $I_3$ as in the proof of Lemma \ref{l:reg. 1,alpha}.
Since
     \begin{align*}
         &|u(x_1+h)-u(x_1)-\nabla u(x_1)\cdot h-u(x_2+h)+u(x_2)+\nabla u(x_2)\cdot h|\\
         &\le\sum_{i=1,2}\left|u(x_i+h)-u(x_i)-\nabla u(x_i)\cdot h-\frac12 Hu(x_i)(h,h)\right|\\
         &+\frac 1 2\left|Hu(x_1)(h,h)-Hu(x_2)(h,h)\right|\\
         &\le 2[u]_{2,\alpha,\R^d}|h|^{2+\alpha}+\frac12 [u]_{2,\alpha,\R^d}|x_1-x_2|^\alpha |h|^2,
     \end{align*}
     it follows that
     \begin{align}\label{eq:reg. 2,alpha no1}
         I_1\le c_2(d,\alpha,\phi)[u]_{2,\alpha,\R^d}\big(r^{2+\alpha-2\delta_2}+|x_1-x_2|^\alpha r^{2-2\delta_2}\big).
     \end{align}
Analogously, the bound
 \begin{align}\label{eq:reg. 2,alpha no2}
         I_2&\le c_3(d,\alpha,\phi)[u]_{2,\alpha,\R^d}\big(|x_2-x_1|^{2+\alpha} r^{-2\delta_2}+|x_2-x_1|^2 r^{\alpha-2\delta_2}\\
         &\hspace{12em}+|x_2-x_1|^{1+\alpha} r^{1-2\delta_2}+|x_2-x_1|r^{1+\alpha-2\delta_2}\big)
     \end{align}
     follows by noting that
     \begin{align*}
         &|u(x_1+h)-u(x_1)-\nabla u(x_1)\cdot h-u(x_2+h)+u(x_2)+\nabla u(x_2)\cdot h|\\
         &\le\left|\big(u(x_1+h)-u(x_2+h)-\nabla u(x_2+h)\cdot ({x_1-x_2})-\frac12 Hu(x_2+h)({x_1-x_2},{x_1-x_2})\big)\right.\\
         &\hspace{2em}-\big(u(x_1)-u(x_2)-\nabla u(x_2)\cdot ({x_1-x_2})-\frac12 Hu(x_2)({x_1-x_2},{x_1-x_2})\big)\\
         &\hspace{4em}+\left(\frac12 Hu(x_2+h)({x_1-x_2},{x_1-x_2})-\frac12 Hu(x_2)({x_1-x_2},{x_1-x_2})\right)\\
         &\hspace{6em}+\big(\nabla u(x_2+h)\cdot ({x_1-x_2})-\nabla u(x_2)\cdot ({x_1-x_2})-Hu(x_2)({x_1-x_2},h)\big)\\
         &\hspace{8em}\left.-\big(\nabla u(x_1)\cdot h-\nabla u(x_2)\cdot h-Hu(x_2)({x_1-x_2},h)\big)\right|\\
         &\le (2|x_2-x_1|^{2+\alpha}+|x_2-x_1|^{2}|h|^\alpha+|x_2-x_1||h|^{1+\alpha}+|x_2-x_1|^{1+\alpha}|h|)[u]_{2,\alpha,\R^d}.
     \end{align*}     
     The bound for $I_3$ and the concluding remarks follow by the same argument as in the proof of Lemma \ref{l:reg. 1,alpha}.
\end{proof}

\begin{proof}[Proof of Theorem \ref{t:regularity SBM}]
    Let $f\in C^{k+\alpha}(\R^d)$ and let $u$ solve \eqref{eq:distr.linear solution}. We investigate the global regularity of $u$ by obtaining estimates of $u$ and its derivatives in balls $B(x,1)$, $x\in \R^d$, with the constants independent of $x$. Therefore, without loss of generality, one can take $x=0$.

    Similarly as in \cite{Silvestre2007regularity}, pick a nonnegative function $\eta\in C_c^\infty(\R^d)$ such that $0\le \eta \le 1$, $\eta\equiv 1$ in $B(0,1)$ and $\eta\equiv 0$ in $B(0,2)^c$, and define
    \begin{align*}
        u_0\coloneqq G^{\phi}_{\R^d}\big(\eta f\big).
    \end{align*}
    Obviously, $\eta f\in C_c^{k+\alpha}(\R^d)$. Also, it is easy to see that $u_0\in L^\infty(\R^d)$ and that we have
    \begin{align}\label{eq:u0 Linf norm}
        \|u_0\|_{L^\infty(\R^d)} \le  \|f\|_{L^\infty(\R^d)} \|G^{\phi}_{\R^d}\big(\1_{B(0,2)}\big)\|_{L^\infty(\R^d)}\le c(d,\phi)\|f\|_{L^\infty(\R^d)}.
    \end{align}
    We further have that $u-u_0$ is harmonic in $B(0,1)$ with respect to $\LoR$, since 
    \[
    \LoR(u-u_0)=(1-\eta)f=0, \qquad \text{in } B(0,1),
    \] 
see e.g. \cite[Proposition 2.5]{Bio22}. By \cite[Theorem 1.7]{grzywny_potential_kernels} we have that  $u-u_0$ is smooth in $B(0,1)$ and there exists $c_1=c_1(d,k,\alpha,\phi)>0$ such that
    \begin{align}\begin{split}\label{eq:reg. SBM bnd1}
    \|u-u_0\|_{C^{k+\alpha+2\delta_1}(B(0,1/2))}&\le c_1\|u-u_0\|_{L^\infty(\R^d)}\\\overset{\eqref{eq:u0 Linf norm}}&{\le}  c_1\big(\|u\|_{L^\infty(\R^d)}+\|f\|_{L^\infty(\R^d)}\big).
    \end{split}
    \end{align}

    In order to estimate the $C^{k+\alpha+2\delta_1}$ norm of $u_0$ define $v_0\coloneqq G^{\phi^*}_{\R^d}u_0$, where $\phi^*$ is the conjugate of $\phi$ defined in \eqref{eq:phi_conjugate}. By Lemma \ref{l: GDf(GDf*)=GD} we have that $v_0=G_{\R^d}(\eta f)$ so the regularity of solutions to the classical Poisson equation implies $v_0\in C^{k+2+\alpha}(\R^d)$. Since $\LozR v_0=u_0$, and since $\phi^*$ satisfies \ref{WSC} with the exponents $1-\delta_2$ and $1-\delta_1$, see \eqref{eq:scaling condition phi*}, by either Lemma \ref{l:reg. 0,alpha}, \ref{l:reg. 1,alpha} or \ref{l:reg. 2,alpha}, we have $u_0\in C^{k+\alpha+2\delta_1}(\R^d)$ and
    \begin{align*}
        \|u_0\|_{C^{k+\alpha+2\delta_1}(\R^d)}\le c_2(d,k,\alpha,\phi)\|v_0\|_{C^{k+2+\alpha}(\R^d)}\le c_3(d,k,\alpha,\phi)\|f\|_{C^{k+\alpha}(\R^d)}.
    \end{align*}
The case when $f\in L^\infty(\R^d)$ follows analogously, using the fact that $v_0\in C^{1,\gamma}(\R^d)$ for every $\gamma\in(0,1)$ and hence $u_0\in C^\beta(\R^d)$ for every $\beta \in (0,2\delta_1)$. 
\end{proof}

\begin{rem}\label{r:Gf indefinite no1}
The assumption $d\ge3$ from the beginning of this article is essential in the proof of Theorem \ref{t:regularity SBM} to conclude that the function $v_0=G_{\R^d}(\eta f)$ is well defined and finite. Indeed, if $d\le 2$ and $f\ge c1_B$ for some ball $B\subset B(0,1)$ and $c>0$, then by the calculations in \cite[Propsition 14.2]{bernstein} we have
    \begin{align*}
        v_0(x)=G^{\phi^*}_{\R^d}u_0(x)=\int_0^\infty P_t(\eta f)(x)dt=\ex_x\left[\int_0^\infty \eta(W_t)f(W_t)dt\right]\equiv \infty,
    \end{align*}
    where the last equality comes from the recurrence of the Brownian motion in $d\le 2$, see e.g. \cite[Theorem 3.27]{morters_peres_BM}. Therefore, the presented approach, used as well in the proof of \cite[Proposition 2.8]{Silvestre2007regularity}, cannot be applied in the case of $d\le 2$.
\end{rem}

Now we prove a regularity theorem for $\Lo$ by applying the same approach as in \cite[Lemma 21]{AbaDupaNonhomo2017}. As the calculations are highly non-trivial, we present the complete proof for clarity, to fill in some gaps, as well as to fix a minor error at the end of the proof of \cite[Lemma 21]{AbaDupaNonhomo2017}.

\begin{thm}\label{t:regularity SKBM}
    Let $\alpha\in(0,1)$ and $k\in  \N_0$ such that $k+\alpha+2\delta_1\notin\N$, and let $f\in C^{k+\alpha}(D)$. If $u\in \LLL$ solves 
    \begin{align}\label{eq:sprectral_linear}
        \Lo u = f \quad\text{in $D$}
    \end{align}
    in the distributional sense, then $u\in C^{k+\alpha+2\delta_1}(D)$. Furthermore, for any given sets $K\subsub K'\subsub D$, there exists a constant $C=C(d,\alpha,K,K',D,\phi)>0$ such that
    \begin{align}\label{eq:t:spectral smooth1}
	|| u ||_{C^{k+\alpha+2\delta_1}(K)} \le C \left(|| f||_{C^{k+\alpha}(K')}+|| u ||_{\LLL}\right).
    \end{align}
    Moreover, if $f\in L^\infty_{loc}(D)$ and $\beta\in(0,2\delta_1)$, then
    \begin{align}\label{eq:t:spectral smooth2}
	|| u ||_{C^{\beta}(K)} \le C \left(|| f ||_{L^\infty(K')}+|| u ||_{\LLL}\right).
    \end{align}
    In particular, if $f=0$, i.e. $u$ is $\Lo$-harmonic, then $u\in C^\infty(D)$ and the equality $\Lo u(x)=0$ holds at every point $x\in D$.
\end{thm}

\begin{proof}
    The main idea of the proof is to compare a distributional solution $u\in \LLL$ to \eqref{eq:sprectral_linear} to a specific solution to
    $\LoR \overline u=\overline f$ in $\R^d$ for a suitable modification $\overline f$ of the extension of $f$ to $\R^d$.

Let $f\in C^\alpha(D)$ and take $K\subsub K'\subsub D$. By \cite[Corollary 2.9]{Bio23} we have that $v\coloneqq \GDfz u$ satisfies $v\in \LLL$ and 
    \begin{align}\label{eq:v in LLL}
        ||v||_{\LLL}\le c_1(d,D,\phi)||u||_{\LLL}.
    \end{align}
    Furthermore, it holds that $-\Delta v=f$ in the distributional sense in $D$. Indeed, for $\psi\in C_c^\infty(D)$, by Fubini's theorem we get
    \begin{align*}
	\int_Dv(x)(-\Delta)\psi(x)dx&=\int_D\GDfz u(x)(-\Delta)\psi(x)dx=\int_Du(x)\GDfz\big((-\Delta)\psi\big)(x)dx\\
	&=\int_D u(x)\Lo \psi(x)dx,
    \end{align*}
	where the last equality follows from \cite[Lemma 2.7 \& Proposition 2.2]{Bio23}. By the classical elliptic regularity, this means that $v\in C^{2,\alpha}(D)$, see \cite[Section 2.2]{RosOtonFernRegularBook}.
	
	Next, we show that  
 \begin{equation}\label{eq:u_by_v}
     u=\int_0^\infty\big(v-P_t^Dv\big)\nu(t)dt.
 \end{equation}
 Recall that for $\psi \in C_c^\infty(D)$ by \cite[Eq. (2.30)]{Bio23} we have
	\begin{align*}
		\Loz \psi(x)=\int_0^\infty \big(\psi(x)-P_t^D\psi(x)\big)\nu(t)dt,
	\end{align*}
	since $C_c^\infty \subset \DD(\left.\Delta\right\vert_{D})$.
	Hence, 
	\begin{align}
		\int_D u(x)\psi(x)dx&=\int_D v(x)\Loz\psi(x)\nonumber\\
		&=\int_Dv(x)\left(\int_0^\infty\big(\psi(x)-P_t^D\psi(x)\big)\nu(t)dt\right)dx.\label{eq:integral change eq1}
	\end{align}
Note that we can change the order of integration in \eqref{eq:integral change eq1}, since $v\in \LLL$ and 
 \begin{equation}\label{eq:fubini_condition}
     \int_0^\infty\big|\psi(x)-P_t^D\psi(x)\big|\nu(t)dt\lesssim \delta_D(x),\ \ x\in D.
 \end{equation}
 Indeed, by \cite[Eq. (2.26)]{Bio23} we have that $\psi=\sum_{j=1}^\infty \wt \psi_j \varphi_j$ where  $|\wt \psi_j|\le c_2(\psi,m,D)\lambda_j^{-m}$, for any $m\in \N$. Hence, by \cite[Eq. (A.35) and the display after]{Bio23} and \eqref{eq:Pt=e{-lt}}, we get that  
	\begin{align*}
		|\psi-P_t^D\psi|&=\big|\sum_{j=1}^\infty \wt \psi_j \varphi_j-\sum_{j=1}^\infty \wt \psi_j P_t^D\varphi_j\big|\le\sum_{j=1}^\infty |\wt \psi_j| |\varphi_j|(1-e^{-\lambda_j t})\\
		&\le c_2(\psi,m,D)\de(x)\sum_{j=1}^\infty \lambda_j^{ -m+d/4+1} (1-e^{-\lambda_j t}).
	\end{align*}

Now \eqref{eq:fubini_condition}	follows by choosing $m$ large enough and noting that by \eqref{eq:Weyl's law} we have that 
\[
\int_0^\infty(1-e^{-\lambda_j t})\nu(t)dt=\phi^*(\lambda_j)\le \lambda_j\asymp j^{2/d}.
\]Applying Fubini's theorem to \eqref{eq:integral change eq1} in the first line, and symmetry of $P_t^D$ in the second, we get that
	\begin{align}
		\int_D u(x)\psi(x)dx&=\int_0^\infty\left(\int_Dv(x)\big(\psi(x)-P_t^D\psi(x)\big)dx\right)\nu(t)dt\nonumber\\
		&=\int_0^\infty\left(\int_D\psi(x)\big(v(x)-P_t^Dv(x)\big)dx\right)\nu(t)dt.\label{eq:integral change eq2}
	\end{align}
	Here $P_t^Dv(x)$ is well defined because of \eqref{eq:heat kernel estimate} and the fact that $v\in \LLL$. Now we want to change the order of integration in \eqref{eq:integral change eq2}. To do so, let $\eta\in C_c^\infty(D)$ such that $0\le \eta \le 1$, $\eta\equiv 1$ in $\{x:\dist(x,\supp \psi)<\dist(\supp \psi, \partial D)/2\}$, and write
	\begin{align*}
		\psi(x)\big(v(x)-P_t^Dv(x)\big)&=\psi(x)\big(\eta(x)v(x)-P_t^D[\eta \,v](x)\big)\\&\hspace{6em}{-}\psi(x)P_t^D[(1-\eta)v](x).
	\end{align*}
	This way $\eta v \in C^{2+\alpha}_c(D)$, hence 
	\begin{align}\label{eq:integrab change eq3}
		||\eta v-P_t^D[\eta \,v]||_{L^\infty(D)}\le\min\{t ||\Delta(\eta v)||_{L^\infty(D)}, 2\,||\eta\, v||_{L^\infty(D)}\}\lesssim 1\wedge t,
	\end{align}
	see e.g. \cite[Eq. (13.3)]{bernstein}. Further, for $x\in \supp \psi$ by \eqref{eq:heat kernel estimate} we have
	\begin{align}\label{eq:integrab change eq4}
		|\psi(x)P_t^D[(1-\eta)v](x)|\lesssim \int_{D} \frac{\de(y)}{t^{d/2+1}}e^{-\frac{c_3}{t}} |v(y)| dy\lesssim 1\wedge t,
	\end{align}
	where $c_3=c_3(\dist(\supp\psi,\partial D),D)>0$ and the constant of comparability is independent of $x\in \supp \psi$. Since $\int_0^\infty(1\wedge t)\nu(t)dt<\infty$, by \eqref{eq:integrab change eq3} and \eqref{eq:integrab change eq4}, we can change the order of integration in \eqref{eq:integral change eq2} to get
	\begin{align*}
			\int_D u(x)\psi(x)dx=\int_D\psi(x)\left(\int_0^\infty\big(v(x)-P_t^Dv(x)\big)\nu(t)dt\right)dx.
	\end{align*}
	Hence \eqref{eq:u_by_v} holds a.e.~in $D$. Moreover, the right-hand side in \eqref{eq:u_by_v} is continuous in $D$ since we can apply the dominated convergence theorem due to \eqref{eq:integrab change eq3} and \eqref{eq:integrab change eq4}. Therefore, there exists a version of $u$ which is continuous in $D$. 
	
    Take now $\overline f\in C_c^\alpha(\R^d)$ such that $\overline f=f$ in  a neighbourhood of $K$, $\supp \overline f\subsub  K' $, and such that $|| \overline f||_{C^\alpha(\R^d)}\le c_4(K,K',D)||f||_{C^\alpha(K')}$. Define $\overline u=G_{\R^d}^\phi \overline f$, and note that $\overline u$ solves $\LoR u=\overline f$ in $\R^d$ in the distributional sense by \cite[Proposition 2.5]{Bio22}. Hence, by Theorem \ref{t:regularity SBM} we get that $\overline u\in C^{\alpha+2\delta_1}(\R^d)$ and that 
    \begin{align}\label{eq:overline u smooth}
        ||\overline u||_{C^{\alpha+2\delta_1}(\R^d)}\lesssim||\overline f||_{C^{\alpha}(\R^d)}\le c_5||f||_{C^\alpha(K')},
    \end{align}
     where $c_5=c_5(d,\alpha,K',D,\phi)>0$. 
     
 It remains to estimate the difference $u-\overline u$ and, as a first step, we prove a representation for $\overline u$ analogous to \eqref{eq:u_by_v}. Define $\overline v\coloneqq G_{\R^d}^{\phi^*}\overline u$ and note that $\overline v$ is well defined,  since $\overline v=G_{\R^d}\overline f$ by Lemma \ref{l: GDf(GDf*)=GD}, and by a simple calculation it holds that 
    \begin{align}\label{eq:overline v integrability no2}
		|\overline v(x)|\le \int\limits_{\supp \overline f}\frac{1}{|x-y|^{d-2}}\,|\overline f(y)|dy\le c_6\|\overline f\|_{C^\alpha(\R^d)}(1\wedge |x|^{-d+2}), \quad x \in \R^d,
	\end{align}
	where $c_6=c_6(d,K',D)>0$.
Further, we get that $(-\Delta)\overline v=\overline f$ in $\R^d$ in the distributional sense. Indeed, for $\psi \in C_c^\infty(\R^d)$ we have
	\begin{align*}
		\int_{\R^d}\overline v(x)(-\Delta)\psi(x)dx&=\int_{\R^d}G_{\R^d}\overline f(x)(-\Delta)\psi(x)dx=\int_{\R^d} \overline f(x)G_{\R^d} \big((-\Delta)\psi\big)(x)dx\\
  &=\int_{\R^d}\overline f(x)\psi(x)dx,
	\end{align*}
	where the last equality follows from \cite[Proposition 2.8]{chung_zhao} since $d\ge 3$. Hence $\overline v \in C^{2+\alpha+2\delta_1}(\R^d)$, see e.g. \cite[Corollary 2.17]{RosOtonFernRegularBook} and notice that, in the same way as in Theorem \ref{t:regularity SBM}, the bound for $\overline v$ in $B(x,\frac12)$ in \cite[Corollary 2.17]{RosOtonFernRegularBook} does not depend on $x$ since $\overline f\in C_c^{\alpha}({\R^d})$. Now, similarly as in \eqref{eq:u_by_v}, we prove that
	\begin{align*}
		\overline u=\int_{0}^\infty\big(\overline v- P_t\overline v\big)\nu(t)dt.
	\end{align*}
	To this end, similarly as before, let $\psi\in C_c^\infty(\R^d)$. We have 
	\begin{align}
		\int_{\R^d} 	\overline u(x)\psi(x)dx&=\int_{\R^d} 	\overline v(x)\LozR\psi(x)\nonumber\\
		&=\int_{\R^d}\overline v(x)\left(\int_0^\infty\big(\psi(x)-P_t\psi(x)\big)\nu(t)dt\right)dx.\label{eq:integral change eq3}
	\end{align}
	Again, we want to change the order of integration in \eqref{eq:integral change eq3}.
	First, note that $\overline v$ is bounded on the bounded set $U\coloneqq \{x\in \R^d:\dist(x,\supp \psi)< 2 \diam (\supp \psi)\}\supset \supp \psi$. Also,
	\begin{align*}
		||\psi-P_t\psi||_{L^\infty(\R^d)}\le \min\{t||\Delta\psi||_{L^\infty(\R^d)},2\,||\psi||_{L^\infty(\R^d)}\}\lesssim 1\wedge t.
	\end{align*}
	Hence,
	\begin{align}\label{eq:overline v integrability no1}
		\int_{U}|\overline v(x)|\left(\int_0^\infty\big|\psi(x)-P_t\psi(x)\big|\nu(t)dt\right)dx<\infty.
	\end{align}
	Second, on $U^c$ we have $\psi=0$, and for $x\in U^c$ and for $y\in \supp \psi$ it holds that
    \begin{align*}
        {1\vee  |x|\asymp|x-y|,}
    \end{align*}
    where the constant depends only on $\supp \psi$ and $U^c$, so
    {\begin{align}\label{eq:proximity inequality0}
        j^*(|x-y|)\asymp j^*(1\vee |x|)=j^*(1)\wedge j^*(|x|)\asymp 1\wedge j^*(|x|),
    \end{align}
    with a comparability constant depending only on $\supp \psi$, $U^c$ and $\phi$. Here, we used monotonicity of $j^*$ and the properties \eqref{eq:slow decr jump kern}.}
    Hence,
	\begin{align}\label{eq:overline v integrability no3}
		&\hspace{-4em}\int_{U^c}|\overline v(x)|\left(\int_0^\infty\big|\psi(x)-P_t\psi(x)\big|\nu(t)dt\right)dx\nonumber\\&\le\int_{U^c}|\overline v(x)|\left(\int_0^\infty P_t|\psi|(x)\nu(t)dt\right)dx\nonumber\\
		&\le \int_{U^c}|\overline v(x)|\int_0^\infty\left( \int_{\R^d}\frac{1}{(4\pi t)^{d/2}}e^{-\frac{|x-y|^2}{4t}}|\psi(y)|dy\right)\nu(t)dt\,dx\nonumber\\
            &=\int_{U^c}|\overline v(x)|\int_{\R^d}|\psi(y)|j^*(|x-y|)dy\, dx\nonumber \\
		&\lesssim  \int_{U^c}|\overline v(x)|\big(1\wedge j^*(|x|)\big)dx<\infty,
	\end{align}
	where in the second to last line we used the definition of the jumping kernel $j^*$, and in the last line the property \eqref{eq:proximity inequality0}, and the fact that $\overline v\in L^1(\R^d,(1\wedge j^*(|x|))dx)$, see \cite[Lemma 2.4]{Bio22}.
	
	Finiteness of the integrals in \eqref{eq:overline v integrability no1} and \eqref{eq:overline v integrability no3} implies that we can change the order of integration in \eqref{eq:integral change eq3} to get
	\begin{align}\label{eq:integral change eq5}
		\int_{\R^d}\overline u(x)\psi(x)dx&=\int_0^\infty\int_{\R^d}\big(\overline v(x)\psi(x)-\overline v(x)P_t\psi(x)\big)dx\, \nu(t) dt\nonumber\\
        &=\int_0^\infty\int_{\R^d}\big(\psi(x)\overline v(x)-\psi(x)P_t\overline v(x)\big)dx\, \nu(t)dt.
	\end{align}
	We now change the order of integration in \eqref{eq:integral change eq5}. This can be justified by repeating the same trick with a bump function $\eta$ when we changed the order of integration in \eqref{eq:integral change eq2}, see also \eqref{eq:integrab change eq3} and \eqref{eq:integrab change eq4}. Thus, we get
	\begin{align*}
		\int_{\R^d} 	\overline u(x)\psi(x)dx&=\int_{\R^d}\int_0^\infty\big(\psi(x)\overline v(x)-\psi(x)P_t\overline v(x)\big)\nu(t) dt\,dx,
	\end{align*}
	i.e. $\overline u=\int_{0}^\infty\big(\overline v- P_t\overline v\big)\nu(t)dt$ a.e. in $\R^d$.
	
	Now we estimate $u-\overline u$. In what follows, we fix a minor error from \cite{AbaDupaNonhomo2017} caused by a miscalculation in the last displayed equation on page 450 in \cite{AbaDupaNonhomo2017}. Note that
	\begin{align*}
		u(x)-\overline u(x)&=\int_0^\infty(v(x)-P_t^Dv(x))\nu(t)dt-\int_0^\infty(\overline v(x)-P_t\overline v(x))\nu(t)dt\\
		&=\int_0^\infty\big(v(x)-\overline v(x)-P_t^D(v-\overline v)(x)\big)\nu(t)dt\\&\hspace{6em}+\int_0^\infty\big(P_t\overline v(x)-P_t^D\overline v(x)\big)\nu(t)dt\\
		&\eqqcolon I_1(x)+I_2(x),
	\end{align*}
    where, as a consequence of the following calculations, both $I_1$ and $I_2$ are well defined.
    
    To estimate the first term $I_1$ note that $w(x,t)\coloneqq P_t^D(v-\overline v)(x)$ solves the heat equation in $D$ with a boundary condition $w(x,0)=(v-\overline v)\in C^{2+\alpha}(D)$. Also, $v-\overline v$ is harmonic in a neighborhood of $K$ in the classical sense, hence smooth in a neighborhood of $K$. However, $v-\overline v\in \LLL$ can explode at the boundary so we need to be careful if we want to apply the classical parabolic regularity theory to $w(x,t)$. 
    
    Take $\eta\in C_c^\infty(D)$ such that $0\le \eta\le 1$, $\eta\equiv 1$ in a neighbourhood of $K$, and $\supp \eta \subset K'$. We have for all $x\in K$
    \begin{align*}
        v(x)-\overline v(x)-P_t^D(v-\overline v)(x)&= \left(v(x)-\overline v(x)-P_t^D\big[\eta(v-\overline v)\big](x)\right)\\&\hspace{5em}-P_t^D\big[(1-\eta)(v-\overline v)\big](x)\\
        &\eqqcolon J^{(1)}_1(x,t)-J^{(1)}_2(x,t).
    \end{align*}
       Since $h\coloneqq \eta (v-\overline v)\in C^{2+\alpha}_c(D)$ we have for $x\in K$
    \begin{align*}
        J^{(1)}_1(x,t)&=h(x)-P_t^Dh(x)\\
        &=h(x)-P_th(x)+\int_D \ex_y\left[p(t-\tau_D,W_{\tau_D},x)\1_{\{\tau_D<t\}}\right]h(y)dy\\
        &=h(x)-P_th(x)+\wt{P_t}^Dh(x).
    \end{align*}
    Here, we set $\wt{P_t}^Dh(x)=\int_D \ex_y\left[p(t-\tau_D,W_{\tau_D},x)\1_{\{\tau_D<t\}}\right]h(y)dy$ which is well defined since $h\in L^\infty(D)$.
    Note that $P_t$ and partial derivatives commute. Also, for $x\in K$, we have $\diam D\ge |x-W_{\tau_D}|\ge \dist(K,\partial D)>0$, so every partial derivative with respect to $x$ satisfies
    \begin{align}\label{eq:part deriv of remainder trans kernel}
    |\partial^\beta \ex_y\left[p(t-\tau_D,x,W_{\tau_D})\1_{\{t>\tau_D\}}\right]|\le c(d,\beta,K,D)(1\wedge t).
    \end{align}
    Hence, by \cite[(13.3)]{bernstein} and \eqref{eq:part deriv of remainder trans kernel} we get
    \begin{align}\label{eq:J^1_1 bound eq1}
        \|J^{(1)}_1(\cdot,t)\|_{C^{2+\alpha}(K)}&\le \|h-P_th\|_{C^{2+\alpha}(K)}+\|\wt{P_t}^Dh\|_{C^{2+\alpha}(K)}\nonumber\\
        &\le c(d,K,D)(1\wedge t)\left(\| h\|_{C^{4+\alpha}(K)}+\|h\|_{L^\infty(D)}\right),
    \end{align}
    where $\| h\|_{C^{4+\alpha}(K)}$ is finite since $h$ is (classically) harmonic in  a neighborhood of $K$, hence smooth.
    Also, by the elliptic regularity, it follows that
    \begin{align}\label{eq:J^1_1 bound eq2}
        \|h\|_{C^{4+\alpha}(K)}\le c\| h\|_{L^1(K')}\le c(d,K,K')\|v-\overline v\|_{\LLL},
    \end{align}
    Recall that $\overline v=G_{\R^d}\overline f$, so
	\begin{align}
		||v-\overline v||_{\LLL}&\le ||v||_{\LLL}+\|\overline v\|_{\LLL}\nonumber\\
         &\le c||u||_{\LLL}+\int_D\int_D G_{R^d}(x,y)|\overline f(y)|\de(x)dx\,dy\nonumber\\
         &\le c(d,K,K',D,\phi) \big(||u||_{\LLL}+||f||_{C^{\alpha}(K')}\big).\label{eq:v-overline v integral bound}
    \end{align}
    Thus, by combining \eqref{eq:J^1_1 bound eq1},  \eqref{eq:J^1_1 bound eq2} and \eqref{eq:v-overline v integral bound}, we get
    \begin{align*}
        \|J^{(1)}_1(\cdot,t)\|_{C^{2+\alpha}(K)}&\le c(d,K,K',D,\phi)(1\wedge t)\big(||u||_{\LLL}+||f||_{C^{\alpha}(K')}\big).
    \end{align*}
    For the term $J^{(1)}_2$, first note that $J^{(1)}_2\in C^\infty(D\times(0,\infty))$. Indeed, since $(1-\eta)(v-\overline v)\in \LLL$, the upper bound in \eqref{eq:heat kernel estimate} implies that $\|J^{(1)}_2(\cdot,t)\|_{L^\infty(D)}\le c(t)$, so by the semigroup property we have
    \begin{align}
    J^{(1)}_2(x,t)&=P_{t-\varepsilon}^D\left(P_\varepsilon^D\big[(1-\eta)(v-\overline v)\big]\right)\nonumber\\
    &=P_{t-\varepsilon}(J^{(1)}_2(\cdot,\varepsilon))-\wt{P}_{t-\varepsilon}^D(J^{(1)}_2(\cdot,\varepsilon)).\label{eq:J_2 smooth}
    \end{align}
    Since $J^{(1)}_2(\cdot,\varepsilon)$ is bounded in $D$, we can differentiate under the integral sign in \eqref{eq:J_2 smooth} to get that $J^{(1)}_2\in C^\infty(D\times(0,\infty)))$.
    Further, since $J^{(1)}_2(x,t)=P_t^D\big[(1-\eta)(v-\overline v)\big](x)$ also solves the heat equation in $D\times (0,\infty)$, by the standard parabolic regularity, see e.g. \cite[Theorem 9 in Section 2.3]{evans_pde}, we have
    \begin{align}
        \sup_{s\in (\frac34 t,t)}\|J^{(1)}_2(\cdot,s)\|_{C^{2+\alpha}(K)}&\le \frac{c(K,K')}{t^{2+\alpha}}\sup_{s\in (\frac12 t,t)}\|J^{(1)}_2(\cdot,s)\|_{L^\infty(K')}\nonumber\\
        &=\frac{c(K,K')}{t^{2+\alpha}}\sup_{s\in (\frac12 t,t)}\|P_s^D\big[(1-\eta)(v-\overline v)]\|_{L^\infty(K')}\label{eq:parabolic eq dimin.1}\\
        &\le c(d,K,K',D)\|v-\overline v\|_{\LLL},\label{eq:parabolic eq dimin.2}
    \end{align}
    where the last inequality comes from \eqref{eq:heat kernel estimate} and the fact that $(1-\eta)(v-\overline v)\equiv 0$ in a neighbourhood of $K'$. In fact, the last inequality holds for any exponent instead of the exponent $2+\alpha$ but with a different constant $c(d,K,K',D)$. Hence, by combining \eqref{eq:parabolic eq dimin.2} with \eqref{eq:v-overline v integral bound} we obtain
    \begin{align*}
        \|J^{(1)}_2(\cdot,t)\|_{C^{2+\alpha}(K)}&\le c(d,K,K',D)(1\wedge t)\big(\|u\|_{\LLL}+\|f\|_{C^\alpha({K'})}\big).
    \end{align*}
    Hence,
    \begin{align}\label{eq:I1 smooth}
        \|I_1\|_{C^{2+\alpha}(K)}&\le c(d,K,K',D)(1\wedge t)(\|u\|_{\LLL}+\|f\|_{C^\alpha({K'})}\big).
    \end{align}
 
In order to estimate $I_2$, note that
 \begin{align*}
		P_t\overline v(x)&-P^D_t\overline v(x)=\int_{\R^d}p(t,x,y)\overline v(y)dy\\
		&\qquad\qquad\qquad-\int_{D}\big(p(t,x,y)-\ex_y\left[p(t-\tau_D,W_{\tau_D},x)\1_{\tau_D<t}\right]\big)\overline v(y)dy\\
		&=\int_{D^c}p(t,x,y)\overline v(y)dy+\int_D\ex_y\left[p(t-\tau_D,W_{\tau_D},x)\1_{\tau_D<t}\right]\overline v(y)dy\\
		&\eqqcolon J^{(2)}_1(x,t)+J_2^{(2)}(x,t),
	\end{align*}
where both of these integrals are well defined, due to \eqref{eq:overline v integrability no2}. First, we prove
\begin{align}\label{eq:J1 smooth}
    |\partial^\beta J^{(2)}_i(x,t)|\le c(d,\beta,K,D)\|\overline f\|_{C^\alpha(\R^d)}(1\wedge t),\quad x\in K,
\end{align}
for $i=1$. If $x\in K$ and $y \in D^c$ we have $|x-y|\asymp 1\wedge |y|$ where the constant of comparability depends only on $d$, $K$ and $D$. Since $p(t,x,y)=(4\pi t)^{-d/2}e^{-|x-y|^2/(4t)}$, we easily see that every partial derivative of $p(t,x,y)$ with respect to $x$ satisfies
\begin{align*}
	|\partial^\beta p(t,x,y)|\le c(d)\,t^{-d/2-b}e^{-\frac{|x-y|^2}{4t}}|x-y|^a\le c(d,\beta,K,D)\,\frac{1\wedge |y|^a}{t^{d/2+b}}e^{-\frac{1\wedge |y|^2}{ct}},
\end{align*} 
for some $b=b(\beta)\ge a=a(\beta)\ge 0$ and $c>0$. 
By using this bound and \eqref{eq:overline v integrability no2} we have
\begin{align}
    \int_{D^c} |\partial^\beta p(t,x,y)||\overline v(y)|dy&\le c(d,\beta,K,D)\|\overline f\|_{C^\alpha(\R^d)}\int_1^\infty \frac{1}{t^{d/2+b}}e^{-\frac{r^2}{4t}}r^{a+1}dr\nonumber\\
    &\hspace{-4em}=c(d,\beta,K,D)\|\overline f\|_{C^\alpha(\R^d)}\int_{1/\sqrt{t}}^\infty \frac{1}{t^{d/2+b-a/2-1}}e^{-\frac{h^2}{4}}h^{a+1}dh.\label{eq:int heat diff bnd}
\end{align}
For small $t<1$, the right-hand side of \eqref{eq:int heat diff bnd} is bounded by $t$, since by L'H\^{o}pital's rule we have
\begin{align*}
    \lim_{t\to 0}\frac{\displaystyle \int_{1/\sqrt{t}}^\infty \frac{1}{t^{d/2+b-a/2-1}}e^{-\frac{h^2}{4}}h^{a+1}dh}{t}=\lim_{t\to 0}\frac{\displaystyle \int_{1/\sqrt{t}}^\infty e^{-\frac{h^2}{4}}h^{a+1}dh}{t^{d/2+b-a/2-2}}=0,
\end{align*}
and for large $t\ge 1$, the right-hand side of \eqref{eq:int heat diff bnd} is trivially bounded.
This allows us to differentiate $J^{(2)}_1(x,t)$ under the integral sign to obtain \eqref{eq:J1 smooth} for $i=1$.\\
Similarly, when differentiating $J^{(2)}_2(x,t)$ we may differentiate under the integral sign to get that $J^{(2)}_2(x,t)$ is smooth and that \eqref{eq:J1 smooth} holds for $i=2$. Note that this is justified by relation \eqref{eq:part deriv of remainder trans kernel} and the fact that $\overline v\in L^\infty(\R^d)$. \\
In other words, by integrating \eqref{eq:J1 smooth} with respect to $\nu(t)dt$, we get 
\begin{align}\label{eq:I2 smooth}
            ||I_2||_{C^{2\delta_1+\alpha}(K)}\le c(d,\alpha,K,K',D,\phi)||f||_{C^{\alpha}(K')}.
	\end{align}
 
By combining the bounds \eqref{eq:I1 smooth} and \eqref{eq:I2 smooth}, it follows that
\begin{align*}
    \|u-\overline{u}\|_{C^{\alpha+2\delta_1}(K)}\le C\big(||f||_{C^{\alpha}(K')}+||u||_{\LLL}\big).
\end{align*}
Together with \eqref{eq:overline u smooth} this implies \eqref{eq:t:spectral smooth1}. This completes the proof of \eqref{eq:t:spectral smooth1} for $k=0$, and the proof for $k\in \N$ has trivial differences.

The second part of the theorem, i.e. \eqref{eq:t:spectral smooth2}, follows by analogous calculations and here we emphasize only the non-trivial differences. If $f\in L_{loc}^\infty(D)$, then $v\in C^{1,\beta}(D)$, for every $\beta\in (0,1)$, see \cite[Section 2.2]{RosOtonFernRegularBook}. This implies $\eta v\in C^{1,\beta}_{c}(D)$, which is not smooth enough for $\eta v$ to be in $\DD(\Delta)$, which would allow us to apply \cite[Eq. (13.3)]{bernstein} to get \eqref{eq:integrab change eq3}. To circumvent this issue, note that
\begin{align*}
    \eta v(x)-P_t^D\eta v(x)&=\eta v(x)-P_t\eta v(x)\\&\hspace{2em}+\int_D \ex_y\left[p(t-\tau_D,x,W_{\tau_D})\1_{\{t>\tau_D\}}\right]\eta(y)v(y)dy.
\end{align*}
Now one can easily show, by the arguments displayed above, that 
\begin{align}\label{eq:semigroup nonsmooth bnd0}
    \left|\int_D \ex_y\left[p(t-\tau_D,x,W_{\tau_D})\1_{\{t>\tau_D\}}\right]\eta(y)v(y)dy\right|\lesssim 1\wedge t.
\end{align}
Further, since $p(t,x,y)$ is a symmetric  probability density, we have for all $x\in D$
\begin{align}
    |\eta v(x)-P_t\eta v(x)|&=\left|\int_{\R^d}p(t,x,y)\big(\eta(x) v(x)-\eta(y) v(y)+\nabla(\eta v)(x)\cdot(x-y)\big) dy  \right|\nonumber\\
    &\lesssim \int_{\R^d}p(t,x,y)|x-y|^{1+\beta}dy\nonumber\\
    &\asymp t^{\beta/2+1/2}\int_{0}^\infty e^{-h^2}h^{d+\beta}dh\lesssim t^{\frac{\beta+1}{2}}.\label{eq:semigroup nonsmooth bnd}
\end{align}
Here, we used the H\"older property for $\eta v$ in the second line.
For $t\ge 1$ 
\begin{align}\label{eq:semigroup nonsmooth bnd1}
    \|\eta v-P_t\eta v\|_{L^\infty(D)}\le  2\|\eta v\|_{L^\infty(D)}\lesssim 1.
\end{align}
By combining \eqref{eq:semigroup nonsmooth bnd0}, \eqref{eq:semigroup nonsmooth bnd}, and \eqref{eq:semigroup nonsmooth bnd1}, we obtain
\begin{align}\label{eq:nonsmooth final bnd}
    \|\eta v-P_t^D\eta v\|_{L^\infty(D)}\lesssim 1\wedge t^{\frac{\beta+1}{2}},
\end{align}
for all $\beta \in (0,1)$. Further, since we assume \ref{WSC}, by \cite[Eq. (2.13)]{ksv_minimal2016}, \eqref{eq:scaling and the derivative} and \eqref{eq:scaling condition phi*} it holds that  
\begin{align*}
    \nu(t)\lesssim \phi^*(t^{-1})/t\lesssim t^{-2+\delta_1},\quad t\le 1,
\end{align*}
 hence $\int_0^\infty(1\wedge t^{\wt\beta})\nu(t)dt<\infty$ for $\wt\beta > 1-\delta_1$. Since \eqref{eq:nonsmooth final bnd} holds for any $\beta\in (0,1)$, and in particular for $\beta$ such that $\frac{\beta+1}{2}>1-\delta_1$, we are allowed to change the order of integration in \eqref{eq:integral change eq2}. To conclude the proof, note that for $\overline f\in L_c^\infty(\R^d)$ with $\supp \overline f\subset{K'}$, we have $\overline v\in C^{1,\beta}(\R^d)$, for every $\beta\in (0,1)$, so when we change the order of integration in \eqref{eq:integral change eq5}, we can establish the same bound for $\overline v$ as in \eqref{eq:nonsmooth final bnd}.
\end{proof}

\begin{rem}\label{r:Gf indefinite no2}
    The assumption that $d\ge 3$ comes into play quite a few times in the proof of Theorem \ref{t:regularity SKBM}, e.g. in \eqref{eq:overline v integrability no2} and in the limit after \eqref{eq:int heat diff bnd}. However, the biggest impact of this assumption is in the function $\overline v$ which is not well defined when $d\le2$. Indeed, if $\overline f\ge0$ and $\overline f \not\equiv 0$, $\overline v$ is identically infinite, see Remark \ref{r:Gf indefinite no1}. Therefore, the presented approach, used as well in the proof of \cite[Lemma 21]{AbaDupaNonhomo2017}, cannot be applied in the case of $d\le 2$.
\end{rem}

\section{Large solutions}\label{s:large}

In this section we solve the following problem
\begin{equation}\label{eq:large problem}
	\begin{array}{rcll}
		\Lo u&=& - f(u)& \quad \text{in } D,\\
		\frac{u}{\PDFI\sigma}&=&\infty&\quad \text{on }\partial D,
	\end{array}
\end{equation}
in the distributional and in the pointwise sense, where $f$ is a nonnegative $C^{1}(\R)$ function such that $f(0)=0$, $f(t)>0$ for $t>0$, and $f$ satisfies the following condition:
\begin{assumption}{F}{}\label{F}
	There exist $0<m\le M<\infty$ such that
	$$ (1+m)f(t)\le tf'(t)\le (1+M)f(t),\quad t\in \R.$$
\end{assumption}
Note that, in essence, we are looking at generalisations of the classical superlinearity $f(t)=t^p$, $p>1$. Further, the condition \ref{F} implies that $f$ is nondecreasing, and that the functions $t\mapsto f(t)t^{-(1+M)}$ and $t\mapsto f(t)t^{-(1+m)}$ are  nonincreasing and nondecreasing, respectively. The assumption $f(x)>0$, $x>0$, is not essential. More precisely, we can instead assume $f\not\equiv 0$ and apply minimal technical changes at the beginning of the Subsection \ref{ss:construction large}, in particular, to the definition of the function $\varphi$ near the origin.

In order to obtain the solution, we mimic the approach for the spectral fractional Laplacian and the power nonlinearity in \cite{AbaDupaNonhomo2017}.
The main idea of the construction of a large solution is to build an auxiliary increasing sequence of weak solutions to \eqref{eq:large problem} but with boundary conditions that are finite and increasing, and to find a supersolution to \eqref{eq:large problem} (with an infinite boundary condition) which dominates the aforementioned auxiliary sequence of weak solutions. The goal is then to show that the auxiliary sequence converges to a solution to \eqref{eq:large problem}.

With this objective in mind, consider the sequence of semilinear problems:
\begin{equation}\label{eq:semi-approx}
	\begin{array}{rcll}
		\Lo u_j&=& - f(u_j)& \quad \text{in } D,\\
		\frac{u_j}{\PDFI\sigma}&=&j&\quad \text{on }\partial D,
	\end{array}
\end{equation}
where $j\in \N$ and where we assume that the nonlinearity $f$ satisfies the condition 
\begin{equation}\label{eq:integrability}
    \int_0^1 f\left(\frac{1}{s^2\phi(s^{-2})}\right)s\, ds=\frac 1 2 \int_0^1 f\left(\frac{1}{s\phi(s^{-1})}\right)\, ds<\infty.
\end{equation}
Since the nonlinearity $f$ is a nondecreasing function, the condition \eqref{eq:integrability}, together with \eqref{eq:PDFI_sigma} and \cite[Theorem 5.10]{Bio23}, implies that the problems \eqref{eq:semi-approx} have unique weak solutions
\begin{align*}
	u_j=-\GDFI f(u_j)+j\PDFI\sigma,\quad j\in \N,
\end{align*}
which are also continuous in $D$ and nonnegative. In the following lemma, we prove that $(u_j)_j$ is an increasing sequence of functions. Note that this claim does not immediately follow by the argument in \cite[Chapter 7]{AbaDupaNonhomo2017} based on the increasing boundary conditions in \eqref{eq:semi-approx}.  This is due to the fact that Kato's inequality \cite[Proposition 5.4]{Bio23}, cf. \cite[Lemma 29]{AbaDupaNonhomo2017}, holds only for the zero boundary condition problem.

\begin{lem}\label{l:approx incre}
	The sequence $(u_j)_j$ increases as $j\to\infty$.
\end{lem}
\begin{proof}
	To prove the monotonicity of the sequence $(u_j)_j$ we first present the construction of the solution to \eqref{eq:semi-approx} from \cite[Theorem 5.10]{Bio23}. The idea is to approximate the problem \eqref{eq:semi-approx} with a sequence of auxiliary zero boundary Dirichlet problems. Let $(\wt f_k)_k$ be a sequence of nonnegative, continuous and bounded functions on $D$ such that $\GDFI\wt f_k\uparrow \PDFI (\sigma)$ and let $u_j^{(k)}$, $j\in\N$, be a solution to the following linear problem
	\begin{align}\label{eq:semi non-positive fk approx}
		\begin{cases}
			\Lo u_j^{(k)}=-f\big(u_j^{(k)}\big)+j\wt f_k,&\textrm{in $D$},\\
			\frac{u_j^{(k)}}{\PDFI\sigma}=0,&\textrm{on $\partial D$}.
		\end{cases}
	\end{align}
	Then, by using the Arzel\`{a}-Ascoli theorem, we can find an appropriate subsequence of $(u_j^{(k)})_k$ that converges to the solution $u_j$ to \eqref{eq:semi-approx}. For details we refer to \cite[Theorem 5.10]{Bio23}. Next, note that the function $w\coloneqq u_{j}^{(k)}-u_{j+1}^{(k)}$ is a weak solution to the linear problem
	\begin{align*}
		\begin{cases}
			\Lo w=f\big(u_{j+1}^{(k)}\big)-f\big(u_{j}^{(k)}\big)-\wt f_k,&\textrm{in $D$},\\
			\frac{w}{\PDFI\sigma}=0,&\textrm{on $\partial D$}.
		\end{cases}
	\end{align*}
	Since $f$ is nondecreasing and $\wt f_k$ is nonnegative, by Kato's inequality \cite[Proposition 5.4]{Bio23}, we have
	\begin{align*}
		w^+=\big(u_{j}^{(k)}-u_{j+1}^{(k)}\big)^+&\le \GDFI\left(\big(f\big(u_{j+1}^{(k)}\big)-f\big(u_{j}^{(k)}\big)-\wt f_k\big)\1_{\{u_{j}^{(k)}\ge u_{j+1}^{(k)}\}}\right)\le 0,
	\end{align*}
	which implies that $u_{j}^{(k)}\le u_{j+1}^{(k)}$ in $D$.
	
	Now we use Cantor's diagonal argument to show that $u_j\le u_{j+1}$ in $D$. By the Arzel\`{a}-Ascoli theorem, see the proof of \cite[Theorem 5.10]{Bio23}, there is a subsequence $\big(u_{1}^{(k_n^{(1)})}\big)_n$ of $\big(u_{1}^{(k)}\big)_k$ which converges to $u_1$. We have shown that for every $k_n^{(1)}$, $n\in \N$, it holds that $u_{j}^{(k_n^{(1)})}\le u_{j+1}^{(k_n^{(1)})}$, for all $j\in \N$. Analogously, there is a subsequence $\big(u_{2}^{(k_n^{(2)})}\big)_n$ of $\big(u_{2}^{(k_n^{(1)})}\big)_n$ which converges to $u_2$. By continuing this procedure we arrive to a subsequence $(k_n^{(n)})_{n}$ for which we have
	\begin{align*}
		\lim_{n\to\infty} u_{j}^{(k_n^{(n)})}&=u_j,\quad j\in\N,\\
		u_{j}^{(k_n^{(n)})}&\le u_{j+1}^{(k_n^{(n)})},\quad j\in \N.
	\end{align*}
	Hence, $u_j\le u_{j+1}$ in $D$.
\end{proof}

 The following result justifies the consideration of the pointwise boundary condition in \eqref{eq:semi-approx}. Additionally, we consider further regularity of the solution when $f\in C^\alpha$. Recall that the reason for working in the setting of classical H\"older regularity, instead of generalised H\"older spaces corresponding to $\phi$ as in \cite{BaeKassmann} and \cite{KimLee}, is that we want to obtain regularity results for all ranges of Matuszewska indices in \eqref{eq:simple global scaling}.

\begin{lem}\label{l:approx sol. point.}
	It holds that $\lim\limits_{x\to\partial D}\frac{u_j(x)}{\PDFI\sigma(x)}=j$. Further, if {$f\in C^{1\vee (2\delta_2-2\delta_1+\varepsilon)}(\R)$, for some $\varepsilon>0$,  then $u_j$ is a pointwise solution to \eqref{eq:semi-approx}.}
\end{lem}
\begin{proof}

	Since $u_j\le j \PDFI\sigma$, by monotonicity of $f$, the condition \ref{F} and \eqref{eq:PDFI_sigma}, it follows that \begin{align*}
		0\le \GDFI\left(f(u_j)\right)\le \GDFI(f(j\PDFI\sigma))\le c(d,D, M, j) \GDFI\left(f(\de^{-2}\phi(\de^{-2})^{-1})\right).
	\end{align*}
	Since the condition \ref{F} implies that $t\mapsto f(t)t^{-(1+M)}$ is nonincreasing, one can easily check that the function $t\mapsto f(t^{-2}\phi(t^{-2})^{-1})$ satisfies the conditions \textbf{(U1)-(U4)} from \cite[Theorem 3.6]{Bio23}. This in turn implies that
	\begin{align*}
		0\le\lim_{x\to\partial D}\frac{\GDFI\left(f(u_j)\right)(x)}{\PDFI\sigma(x)}\lesssim \lim_{x\to\partial D}\frac{\GDFI\left(f(\de^{-2}\phi(\de^{-2})^{-1})\right)(x)}{\PDFI\sigma(x)}=0.
	\end{align*}
	Therefore,
	$$\lim_{x\to\partial D}\frac{u_j(x)}{\PDFI\sigma(x)}=\lim_{x\to\partial D}\frac{\PDFI(j\sigma)(x)}{\PDFI\sigma(x)}=j.$$
    
For the second part of the lemma, recall that by \cite[Theorem 5.10]{Bio23} $u_j$ is continuous, hence locally bounded. Also, recall that a weak solution is also a distributional one. Since {$f\in C^{1\vee (2\delta_2-2\delta_1+\varepsilon)}(\R)$}, we have $f(u_j)\in L^\infty_{loc}(D)$, so Theorem \ref{t:regularity SKBM} implies $u_j\in C^\beta(D)$ for every $\beta\in(0,2\delta_1)$.

{
Now we use a bootstrap argument. First note that that $f(u_j)\in C^{\gamma}(D)$ for $\gamma=\big(1\vee (2\delta_2-2\delta_1+\varepsilon)\big)\wedge \beta$, for all $\beta\in(0,2\delta_1)$, so by Theorem \ref{t:regularity SKBM}  we get $u_j\in C^{\gamma+2\delta_1}(D)$. This reasoning may be repeated until we reach that $f(u_j)\in C^{\gamma}(D)$ for some $\gamma>2\delta_2-2\delta_1$. Now Theorem \ref{t:regularity SKBM} implies $u_j$ is at least in $C^{2\delta_2+\wt\varepsilon}(D)$ for some $\wt \varepsilon>0$, so $u_j$ is a pointwise solution to \eqref{eq:semi-approx} by Remark \ref{r:pointwise solution}.}

\end{proof}

{
\begin{rem}
    The higher regularity statement in the second part of Lemma \ref{l:approx sol. point.} can be extended to general $f\in C^\alpha(\R)$.  More precisely, if $u$ is a locally bounded distributional solution to $\Lo u=-f(u)$, where $f\in C^\alpha(\R)$ for $\alpha\in (0,\infty)\setminus \N$ such that $\alpha+2\delta_1\notin \N$, then $u\in C^\beta(D)$ for 
    \begin{enumerate}[$(i)$]
        \item any $\beta\in\left(0,\frac{2\delta_1}{1-\alpha}\right)$, when $\alpha+2\delta_1<1$;
        \item $\beta=\alpha+2\delta_1$, when $\alpha+2\delta_1>1$.
    \end{enumerate}
    Indeed, in the case $(i)$, by applying the bootstrap argument we get that $u\in C^\gamma(D)$ for $\gamma=\frac{2 \delta_1}{1-\alpha}-\alpha^{k+1}(\frac{2\delta_1}{1-\alpha}-\beta)$, for all $\beta\in(0,2\delta_1)$ and for all $k\in \N$, from which the claim follows. 

    On the other hand, in the case $(ii)$, by the bootstrap argument, at some point we obtain that $u\in C^\gamma(D)$ for $\gamma=\frac{2 \delta_1}{1-\alpha}-\alpha^{k+1}(\frac{2\delta_1}{1-\alpha}-\beta)>1$, which in the following step implies $u\in C^{\gamma\wedge \alpha+2\delta_1}(D)$. As a continuation of this argument, we obtain $u\in C^{ \alpha+2\delta_1}(D)$.
    
In the setting of this section and, specifically in Lemma \ref{l:approx sol. point.}, we rely on Assumption \ref{F} which presumes $f\in C^1(\R)$ under which the existence of the solution is established.
\end{rem}
}
\subsection{Construction of a supersolution}\label{ss:construction large}
A key step in obtaining a solution to \eqref{eq:large problem} is the construction of a supersolution to the same problem. In this subsection, we find a Keller-Osserman-type condition that guarantees that such a supersolution exists. Let us first introduce several functions that are used to construct this supersolution. Recall that the function $f\in C^{1}$ satisfies the condition \ref{F}. Define the antiderivative of $f$ by
\begin{align*}
	F(t)=\int_0^tf(s)ds,\quad t>0,
\end{align*}
and set $\varphi:(0,\infty)\to (0,\infty)$ as
\begin{align*}
	\varphi(t)=\int_t^\infty \frac{ds}{\sqrt{F(s)}},\quad t>0.
\end{align*}
The function $\varphi$ is monotone decreasing so it possesses a decreasing inverse
\begin{equation}\label{eq:psi_definition}
\psi=\varphi^{-1,}
\end{equation}
which also satisfies  
\[
\lim\limits_{t\downarrow 0}\psi(t)=\infty,\ \lim\limits_{t\uparrow\infty}\psi(t)=0.
\]
Note that \ref{F} implies that
\begin{equation}\label{eq:f2}
    \varphi(t)\asymp \sqrt{\frac{t}{f(t)}},\qquad t>0,
\end{equation}
and
\begin{equation}\label{eq:psi1}
	\frac{t^2\psi''(t)}{\psi(t)}\asymp \frac{t^2\psi'(t)^2}{\psi(t)^2}\asymp 1,\qquad t>0.
\end{equation}
Furthermore, the assumption \ref{F} can be also written in terms of $\varphi$ or $\psi$, that is 
\begin{equation}\label{eq:varphi1}
\frac m 2 \frac{\varphi(t)}{t}\leq |\varphi'(t)|\leq \frac M 2 \frac{\varphi(t)}{t},\qquad t>0,
\end{equation}
and
\begin{equation}\label{eq:psi2}
\frac 2 M \frac{\psi(t)}{t}\leq |\psi'(t)|\leq \frac 2 m\frac{\psi(t)}{t},\qquad  t>0,
\end{equation}
see \cite[Remark 1.1]{Aba17}.

Recall that $V:[0,\infty)\to [0,\infty)$ is the renewal function defined in \eqref{eq:V}. In the following, we construct a supersolution to \eqref{eq:large problem} with the boundary behaviour comparable to the function 
\begin{equation}\label{eq:U}
U(x)=\psi(V(\de(x))),\qquad x\in D.    
\end{equation}
Similarly as in \cite[Lemma 3.1]{semilinear_cvw} the following Keller-Osserman-type condition implies the required integrability condition for the supersolution.

\begin{lem}\label{l:supersol integrability}
    The function $U$ from \eqref{eq:U} satisfies $U\in\LLL$ if and only if the following condition holds:
\begin{equation}\label{eq:KO1}
\int_1^\infty \frac{dt}{\phi^{-1}(\varphi(t)^{-2})}<\infty.
\end{equation}
\end{lem}
\begin{proof}
 For $\eta>0$ set $D_\eta=\{x\in D:\ \delta_D(x)<\eta\}$. 
Since  $\psi,V\in L_{loc}^\infty(0,\infty)$ and $\delta\in L^\infty(D)$, it follows that $U\in L^\infty_{loc}(D)$. Therefore, it is enough to show that \eqref{eq:KO1} holds if and only if $U\in L^1(D_{\eta_0},\delta_D)$, for some small $\eta_0>0$. By applying the coarea formula in the first equality and using \eqref{e:V-phi} several times, we arrive to
\begin{align*}
\int_{D_{\eta_0}}U(x)\delta_D(x) dx&=\int_0^{\eta_0} dt\int_{\{x\in D:\delta(x)=t\}}\psi(V(t))t\,d\mathcal H^{d-1}(dx)\asymp \int_0^{\eta_0}\psi(\Phi(t))t\, dt\\
&= \int_{\psi(\Phi(\eta_0))}^\infty s\,\Phi^{-1}(\varphi(s))(\Phi^{-1})'(\varphi(s))|\varphi'(s)|ds\\
\overset{\eqref{eq:varphi1}}&{\asymp} \int_{\psi(\Phi(\eta_0))}^\infty \Phi^{-1}(\varphi(s))(\Phi^{-1})'(\varphi(s))\varphi(s)ds\\
&=\int_{\psi(\Phi(\eta_0))}^\infty \frac{\Phi^{-1}(\varphi(s))\varphi(s)}{\Phi'(\Phi^{-1}(\varphi(s)))} ds\asymp \int_{\psi(\Phi(\eta_0))}^\infty (\Phi^{-1}(\varphi(s)))^2 ds\\
&\asymp\int_{\psi(\Phi(\eta_0))}^\infty \frac{ds}{{\phi^{-1}(\varphi(s)^{-2}))}},
\end{align*}
where in the second to last approximate equality we used $\Phi'(t)\asymp \frac{\Phi(t)}{t}$, for $0<t<1$, see \eqref{eq:scaling and the derivative}. The result now follows by \eqref{eq:KO1}. 
\end{proof}

Next, we prove that, under a slightly stronger Keller-Osserman-type condition,  $U$ has a certain behaviour near the boundary of $D$ suitable for the construction of a supersolution. 

\begin{lem}\label{l:supersol lem1}
 If 
 \begin{equation}\label{eq:KO2}
\int_r^\infty \frac{dt}{\phi^{-1}(\varphi(t)^{-2})}\lesssim\frac{r}{\phi^{-1}(\varphi(r)^{-2})},\qquad r\ge 1,
\end{equation}
 then there exist constants $C>0$ and $\eta>0$ such that
	\begin{align*}
		\Lo U(x)\ge - Cf(U(x)),\quad x\in D_\eta.
	\end{align*}
\end{lem}

\begin{proof}
     Let us fix $\eta>0$ and $x\in D_{\eta}$, where $D_{\eta}$ is defined as in Lemma \ref{l:supersol integrability}.  By following the approach in \cite[Proposition 3.2]{semilinear_cvw}, we decompose $D=\cup_{i=1}^3D_i$ such that
\begin{align*}
&D_1=\left\{y\in D: \delta_D(y)>\tfrac{3}{2}\delta_D(x)\right\},\\
&D_2=\left\{y\in D: \tfrac{1}{2}\delta_D(x)\leq\delta_D(y)\leq\tfrac{3}{2}\delta_D(x)\right\},\\
&D_3=\left\{y\in D: \delta_D(y)<\tfrac{1}{2}\delta_D(x)\right\}.
\end{align*}
Note that $U\in\LLL$ by \eqref{eq:KO2} and Lemma \ref{l:supersol integrability}. Furthermore, $U\in C^{1,1}(D)$, since $\psi$ and $V$ are twice differentiable by \ref{F} and by \cite[Lemma 2.5]{KKLL}, respectively, and that $\de$ is a $C^{1,1}(\overline D)$ function by \cite[Theorem 8.4 in Chapter 7]{DZ}. Therefore, by \eqref{eq:Lo pointwisely} and the fact that $t\mapsto\psi(V(t))$ is nonincreasing we have that
\begin{align*}
-\Lo U(x)&=\textrm{P.V.}\int_D (U(y)-U(x))J_D(x,y) dy -\kappa(x)U(x)\\
&\le \textrm{P.V.}\int_{D_2} (U(y)-U(x))J_D(x,y) dy+\int_{D_3} (U(y)-U(x))J_D(x,y) dy\\
&=:J_2+J_3.
\end{align*}

First, split $J_2$ into two parts
\begin{align*}
    J_2&=\textrm{P.V.}\int\limits_{B_x}(U(y)-U(x))J_D(x,y) dy+\!\!\int\limits_{D_2\setminus B_x}(U(y)-U(x))J_D(x,y) dy\eqqcolon J^1_2+J^2_2,
\end{align*}
where $B_x\coloneqq B(x,\de(x)/2)$. By symmetry, we have
$$\textrm{P.V.}\int_{B_x}\nabla U(x)\cdot(y-x)\left(\int_0^\infty p(t,x,y)\mu(t)dt\right)dy=0,$$
which, together with \eqref{e:J_D} and \eqref{eq:trans.dens. KBM}, implies that
\begin{align*}
    J^1_2&=\textrm{P.V.}\int_{B_x}\big(U(y)-U(x)-\nabla U(x)\cdot(y-x)\big)J_D(x,y) dy\\
    &\qquad +\textrm{P.V.}\int_{B_x}\nabla U(x)\cdot(y-x)\left(\int_0^\infty \ex_x[p(t-\tau_D,y,W_{\tau_D})\1_{\{t>\tau_D\}}]\mu(t)dt\right) dy.
\end{align*}
Since $J_D(x,y)\lesssim j(|x-y|)$ by \eqref{eq:J_D estimate}, the first term in $J^1_2$ can be dealt with in the same way as in \cite{semilinear_cvw} in order to obtain that
\begin{align*}
    \left|\int_{B_x}\big(U(y)-U(x)-\nabla U(x)\cdot(y-x)\big)J_D(x,y) dy\right|\lesssim \psi(V(\de(x)))V(\de(x))^{-2}.
\end{align*}
For the second term in $J^1_2$, note that $\nabla U=\psi'(V(\de))V'(\de)\nabla \de$, and
\begin{align*}
    \max_{t\in (0,\infty)}\frac{p(t,y,W_{\tau_D})}{1\wedge t}\le c(d)\max_{t\in (0,\infty)}\frac{t^{-d/2}e^{-\de(x)^2/8t}}{1\wedge t}\le c(d,D) \de(x)^{d+2},    
\end{align*}
since $|y-W_{\tau_D}|\ge\de(y)\ge \frac12 \de(x)$. This, together with \eqref{eq:varphi1} and \eqref{eq:H1}, implies that
\begin{align*}
    &\left|\int_{B_x}\nabla U(x)\cdot(y-x)\left(\int_0^\infty \ex_x[p(t-\tau_D,y,W_{\tau_D}\1_{\{t>\tau_D\}}]\mu(t)dt\right) dy\right|\\
    &\qquad\lesssim \psi(V(\de(x)))\de(x)^{-1}\int_0^{\de(x)/2}\frac{h^d}{\de(x)^{d+2}}\asymp \psi(V(\de(x)))\de(x)^{-2}\\
    &\qquad\le \psi(V(\de(x)))V(\de(x))^{-2}.
\end{align*}
Thus, $J^1_2\lesssim \psi(V(\de(x)))V(\de(x))^{-2}$. Since $J_D(x,y)\lesssim j(|x-y|)$, by applying \cite[Eq. (3.6)]{semilinear_cvw}, we obtain the same upper bound for $J^2_2$, and therefore $J_2$.

For $J_3$ we repeat the analogous calculation from \cite[ Proposition 3.2]{semilinear_cvw}. Let $Q\in\partial D $ be the projection of $x$ on $\partial D $, that is $x=Q+\delta(x)\nabla\delta(x)$, and $\gamma=\gamma_Q$ the $ C^{1,1}$ function such that
\begin{equation*}
B(Q,R) \cap D  = \{y = (\widetilde y,y_d) \in B(Q,R) \text{ in } CS_Q : y_d > \gamma(\widetilde y)\},
\end{equation*}
where $CS_Q$ denotes the coordinate system with $Q=0$ and $\nabla\delta(x)=e_d$. Recall that the pair $(R,\Lambda)$, where $\Lambda$ is such that $|\nabla \gamma_Q(\wt x)-\nabla \gamma_Q(\wt y)|\le \Lambda |\wt x-\wt y|$, is called the characteristics of the $C^{1,1}$ set $D$. Set 
\begin{equation*}
\omega=\{y\in D:y=Q_y+\delta(y)\nabla\delta(y),\ Q_y\in B(0,R) \cap\partial D \}
\end{equation*}
and
\begin{equation*}
J_3=\int_{D_3\setminus\omega } (U(y)-U(x))J_D(x,y) dy+\int_{D_3\cap \omega} (U(y)-U(x))J_D(x,y) dy=:J_3^1+J_3^2.
\end{equation*}
First note that, for $ \eta $ small enough (depending only on $D$), we have that $|y-x| > \de(x)\vee \de (y)$ for every $y\in D_3\setminus \omega$ and, therefore, \eqref{eq:J_D estimate} implies that $J_D(x,y)\lesssim \de(x)\de(y)$. Since $U$ explodes at the boundary, it follows that
\begin{align}\label{eq:J31}
J_3^1\lesssim \de(x)\int_{D_3\setminus \omega} U(y)\de(y) dy\leq \de(x)||U||_{L^1(D_{\eta},\de(x)dx)} \lesssim f(U(x)).
\end{align}
Similarly, again by choosing $\eta$ small enough, one can show that there exists a constant $C=C(\Lambda)$ such that
\begin{equation}\label{eq:Qy}
 |y-x|\geq C(|\delta(x)-\delta(y)|+|\overline{Q_y}|), \ \ y\in D_3\cap\omega,
\end{equation}
see \cite[Eq. (3.8)]{semilinear_cvw}. Therefore, by \eqref{eq:J_D estimate} we have that
\begin{align*}
 J_3^2&\leq\int_{D_3\cap w}U(y)J_D(x,y)dy\lesssim\de(x)\int_{D_3\cap w}U(y)j(|y-x|)\de(y)dy\nonumber\\
 &\lesssim\de(x)\int_{D_3\cap w}\psi(V(\delta(y)))\frac{\phi((|\delta(x)-\delta(y)|+|\overline{Q_y}|)^{-2})}{(|\delta(x)-\delta(y)|+|\overline{Q_y}|)^{d+2}}\de(y)dy\nonumber\\
 &\asymp\de(x)\int_{0}^{\delta(x)/2}\psi(V(t))t\int_0^R\frac{\phi((|\delta(x)-t|+s)^{-2})}{(|\delta(x)-t|+s)^{d+2}}s^{d-2}dsdt\nonumber\\
 &\lesssim\de(x)\int_{0}^{\delta(x)/2}\psi(V(t))t\frac{\phi((\delta(x)-t)^{-2})}{|\delta(x)-t|^3}dt\asymp \frac{\phi(\delta(x)^{-2})}{\delta(x)^2}\int_{0}^{\delta(x)/2}\psi(V(t))tdt,
\end{align*}
where in the second line we used \eqref{eq:Qy}, \eqref{eq:sharp bnd jump kern} and the fact that $j$ is decreasing, and in the first term in the fourth line we used \cite[Eq. (A.10) \& (A.15)]{Bio23}.
By a change of variable as in the proof of Lemma \ref{l:supersol integrability}, we get that 
\begin{align}
J_3^2&\lesssim \frac{\phi(\delta(x)^{-2})}{\delta(x)^2}\int_{\psi(V(\de(x)/2)))}^\infty \frac{ds}{{\phi^{-1}(\varphi(s)^{-2}))}}\nonumber\\
\overset{\eqref{eq:KO2}}&{\lesssim} \frac{\phi(\delta(x)^{-2})}{\delta(x)^2}\frac{\psi(V(\de(x))))}{{\phi^{-1}(\varphi(\psi(V(\de(x))))^{-2})}}\nonumber\\
&\asymp\phi(\delta(x)^{-2})\psi(V(\delta(x)))\overset{\eqref{e:V-phi}}{\asymp} \psi(V(\delta(x)))V(\delta(x))^{-2}.\label{eq:J32}\nonumber
\end{align}
The result now follows by applying \eqref{eq:f2}, since 
\begin{equation*}
\frac{\psi(V(\delta(x)))}{V(\delta(x))^2}=\frac{U(x)}{\varphi(U(x))^2}\asymp f(U(x)).
\end{equation*}

\end{proof}

\begin{lem}\label{l:supersol lem2}
	Assume that a function $u\in\LLL$ satisfies $\Lo u\in L^\infty_{loc}(D)$ and 
	\begin{align*}
		\Lo u(x)\ge - Cf(u(x)),\quad \de(x)< \eta ,
	\end{align*}
    for some $C>0$ and $\eta>0$. Then, there exists a function $\overline u\in \LLL$ such that 
	\begin{align*}
		\Lo \overline u(x)\ge - f(\overline u(x)),\quad \text{in $D$.}
	\end{align*}
\end{lem}
\begin{proof}
	Let $ D_\eta \coloneqq \{x\in D: \de(x)< \eta\}$. Take $\lambda>0$, hence
	\begin{align*}
		\Lo (\lambda u)=\lambda \Lo u\ge -\lambda C f(u),\quad \text{in $D_\eta$}.
	\end{align*}
Recall that, by \ref{F}, $t\mapsto f(t)t^{-1-m}$ is nondecreasing, so by choosing $\lambda$ large enough we get $f(\lambda u)\ge \lambda C f(u)$.
    Hence, 
	\begin{align*}
		\Lo (\lambda u)\ge -f(\lambda u),\quad \text{in $D_\eta$}.
	\end{align*}
	Let now $\mu\coloneqq\lambda||\Lo u||_{L^\infty(D\setminus D_\eta)}$ and define $\overline u\coloneqq \mu\GDFI\1 +\lambda u$. By \cite[Proposition 2.14]{Bio23} and by Theorem \ref{t:regularity SKBM} we have $\Lo \GDFI \1=\1$ in $D$ in the distributional and in the pointwise sense. Hence, on $D_\eta$ we have
	\begin{align*}
		\Lo \overline u = \mu + \lambda \Lo u\ge \lambda \Lo u\ge - f(\lambda u)\ge - f(\overline u),
	\end{align*}
	and on $D\setminus  D_\eta$  we have
	\begin{align*}
		\Lo \overline u = \mu + \lambda \Lo u\ge 0\ge - f(\overline u).
	\end{align*}
	In other words, $\Lo\overline u\ge -f(\overline u)$ in $D$.
\end{proof}

\begin{cor}\label{c:large supersolution}
	Let $f$ satisfy \ref{F} and \eqref{eq:KO2}. Then there is a function $\overline u\in \LLL \cap C^{1,1}(D)$ such that 
	\begin{align*}
		\Lo\overline u\ge -f(\overline u),\quad \text{in $D$},
	\end{align*}
	both in the distributional and pointwise sense. Furthermore, if the function $\psi$ from \eqref{eq:psi_definition} satisfies
 \begin{equation}\label{eq:bdry_condition}
 \lim_{s\to 0+}\frac{\psi(s)}{s^2\phi^{-1}(s^{-2})}=\infty,
 \end{equation}
  then 
	\begin{align*}
		\lim_{x\to \partial D}\frac{\overline u(x)}{\PDFI\sigma(x)}=\infty.
	\end{align*}
\end{cor}
\begin{proof}
    By Lemmas \ref{l:supersol lem1} and \ref{l:supersol lem2} we have that $\overline u$ is of the form $\overline u=\mu \GDFI\1+\lambda \psi(V(\de))$.
    The boundary behaviour follows directly from  \eqref{e:V-phi}, \eqref{eq:PDFI_sigma} and \eqref{eq:bdry_condition} by substituting $s=V(\delta(x))$.  
\end{proof}
\begin{rem}
 Note that in the standard fractional setting, i.e. for $\phi(t)=t^{\frac \alpha 2}$, $\alpha\in(0,2)$, and for $f(t)=t^p$, the Keller-Osserman-type conditions \eqref{eq:KO1} and \eqref{eq:KO2} are equivalent to $p>1+\tfrac \alpha 2$. Further, conditions \eqref{eq:integrability} and \eqref{eq:bdry_condition} are equivalent to $p<\tfrac 2{2-\alpha}$. In the more general setting of this article, the conditions \eqref{eq:KO1}, \eqref{eq:KO2}, \eqref{eq:integrability}, and \eqref{eq:bdry_condition} are independent.
\end{rem}

\subsection{Existence of a large solution}\label{ss:existence}
 In what follows we assume conditions \ref{F}, \eqref{eq:integrability}, \eqref{eq:KO2} and \eqref{eq:bdry_condition} on the functions $f$ and $\phi$, as well as $f\in C^{1 \vee (2\delta_2-2\delta_1+\varepsilon)}$ for some $\varepsilon>0$. Let us recall the sequence $(u_j)_j$, where $u_j=-\GDFI f(u_j)+j\PDFI\sigma$ is the weak solution to the approximated problem \eqref{eq:semi-approx}. Note that due to Lemma \ref{l:approx sol. point.} it is also a solution in the pointwise sense. Further, by Lemma \ref{l:approx incre} $(u_j)_j$ increases so the limit
\begin{align}\label{eq:large solution}
	u(x)=\lim_{j\to\infty} u_j(x),\quad x\in D,
\end{align}
so the function $u$ on $D$ is well defined but possibly infinite. However, due to the following uniform upper bound on $(u_j)_j$ in terms of the supersolution $\overline u=\mu \GDFI\1+\lambda \psi(V(\de))$ from Corollary \ref{c:large supersolution}, $u$ is locally bounded on $D$.

\begin{lem}\label{l:large finite}
	For every $j\in \N$ it holds that $0\le u_j\le \overline u$ in $D$. In particular, $0\le u\le \overline u$ and $u\in \LLL\cap L^\infty_{loc}(D)$.
\end{lem}
\begin{proof}
Fix $j\in\N$. By Lemma \ref{l:approx sol. point.} and by Corollary \ref{c:large supersolution} we have that $u_j\le \overline u$ on $D_\eta$ for some $\eta =\eta(j)>0$ small enough and
\begin{align}\label{eq:large finite eq1}
	\Lo(\overline u - u_j)\ge f(u_j)-f(\overline u),
\end{align}
pointwisely in $D$. Since $u_j-\overline u\in C(D)$ and $\lim_{x\to \partial D} u_j-\overline u=-\infty$, there exists a $x_0\in D$ such that $\max_{x\in D}\big(u_j(x)-\overline u(x)\big)=u_j(x_0)-\overline u(x_0)\eqqcolon m$. Suppose that $m> 0$. Then, since $x_0$ is the point of maximum of $u_j-\overline u$, by \eqref{eq:large finite eq1} and the monotonicity of $f$, we have
\begin{align*}
	0&{<} \textrm{\textrm{P.V.}}\int_D\big(u_j(x_0)-\overline u(x_0)-u_j(y)+\overline u(y)\big)J_D(x_0,y)+\kappa(x_0)\big(u_j(x_0)-\overline u(x_0)\big)\\
	&=\Lo\big(u_j-\overline u\big)(x_0){\le  f(\overline u(x_0))-f(u_j(x_0))\le} 0,
\end{align*}
which is a contradiction. Hence, $m\le 0$, i.e. $u_j\le \overline u$ in $D$. {Passing} to the limit, we get $u\le \overline u$ in $D$. Since $\overline u\in \LLL \cap C^{1,1}(D)$, we also have $u\in \LLL \cap L^\infty_{loc}(D)$.
\end{proof}

\begin{thm}\label{t:large solution}
	 The function $u\in \LLL$ defined by \eqref{eq:large solution} is a distributional and a pointwise    solution to the semilinear problem
	\begin{equation}\label{eq:large problem in thm}
		\begin{array}{rcll}
			\Lo u&=& - f(u)& \quad \text{in } D,\\
			\frac{u}{\PDFI\sigma}&=&\infty&\quad \text{on }\partial D.
		\end{array}
	\end{equation}
\end{thm}
\begin{proof}
	Recall that the function $u_j$ is a weak solution to \eqref{eq:semi-approx}, hence also a solution in the distributional sense, so
    \begin{align*}
		\int_D u_j(x) \Lo\xi(x)dx=\int_D f(u_j(x)) \xi(x)dx,\quad \xi\in C_c^\infty(D),\,j\in \N.
	\end{align*}
    The fact that $u$ is the solution in the distributional sense of \eqref{eq:large problem in thm} easily follows from the dominated convergence theorem. Specifically, since $0\le u_j\le u\le \overline{u}$ we have that
	\begin{align}\label{eq:large distrib eq1}
		\int_D u_j(x) \Lo\xi(x)dx\overset{j\to \infty}{\longrightarrow}		\int_D u(x) \Lo\xi(x)dx,
	\end{align}
	where we used that $|\Lo \xi|\lesssim \de$ by \cite[Lemma 2.11]{Bio23} and $\overline{u}\in \LLL$ by Corollary \ref{c:large supersolution}. On the other hand, the convergence
	\begin{align}\label{eq:large distrib eq2}
		\int_D f(u_j(x)) \xi(x)dx\overset{j\to \infty}{\longrightarrow}		\int_D f(u(x)) \xi(x)dx
	\end{align}
	follows from the monotonicity of $f$ and local boundedness of $\overline u$. Since $\Lo u= - f(u)$ in $D$ in the distributional sense and $f(u)\in L^{\infty}_{loc}(D)$ by Lemma \ref{l:large finite}, by a bootstrap of a Theorem \ref{t:regularity SKBM} we get that $u\in C^{2\delta_2+\varepsilon}(D)$. Hence, $u$ solves $\Lo u= - f(u)$ also in the pointwise sense. The required  pointwise  boundary behaviour of $u$ follows directly from
	\begin{align*}
		\liminf_{x\to\partial D}\frac{u(x)}{\PDFI\sigma(x)}\ge \liminf_{x\to\partial D}\frac{u_j(x)}{\PDFI\sigma(x)}=j, \qquad \forall\,j\in\N.
	\end{align*}
    The pointwise boundary behaviour implies the weak $L^1$ boundary behaviour in the sense of \eqref{eq:distri solution boundary}. 
\end{proof}

\begin{rem}
    Note that for the obtained solution $u$, it does not hold that $f(u)\in \LLL$. Indeed, since $u_j+\GDFI\big(f({u_j})\big)=j\PDFI\sigma$, and $\uparrow \lim_ju_j=u<\infty$ and $\lim_j j\PDFI\sigma=\infty$, by the monotonicity of $u_j$ we have $\lim_j\GDFI\big(f({u_j})\big)=\GDFI\big(f({u})\big)=\infty$. Hence, by \cite[Lemma 2.8]{Bio23} it follows $f(u)\not\in \LLL$.
\end{rem}
Our obtained large solution is the minimal solution to \eqref{eq:large problem in thm} as we prove in the next theorem.
\begin{thm}\label{t:minimal large sol}
    The solution $u$ from Theorem \ref{t:large solution} is the minimal locally bounded distributional and pointwise solution to the problem \eqref{eq:large problem in thm} such that the boundary condition on $\partial D$ holds pointwisely.
\end{thm}
\begin{proof}
    Let $U$ be any locally bounded distributional (and by bootstrapping Theorem \ref{t:regularity SKBM} a pointwise) solution to \eqref{eq:large problem in thm}. By repeating the proof of Lemma \ref{l:large finite} we get that $u_j\le U$ in $D$, hence by taking the limit we have $u\le U$.
\end{proof}

In the next theorem, we prove that for every supersolution, there exists a maximal solution that is smaller than the supersolution. However, this does not prove that there are many large solutions, i.e. the important questions of the uniqueness of the solution to \eqref{eq:large problem in thm}, as well as the true maximality (i.e. the existence of a solution which is larger than any other solution) remain open.

\begin{thm}\label{t:maximal large sol}
    For every locally bounded pointwise and distributional supersolution $\overline v$ of the problem \eqref{eq:large problem in thm}, there exist the maximal solution of \eqref{eq:large problem in thm} smaller then $\overline v$.
\end{thm}

\begin{proof}
    The key ingredient in the proof is Zorn's lemma. Let $\overline v \in \LLL$ be a locally bounded pointwise and distributional supersolution of \eqref{eq:large problem in thm} and 
    $$\KK=\{v\in \LLL:v\le \overline v\text{ and $v$ solves \eqref{eq:large problem in thm} distributionally}\}.$$
    Note that by the bootstrap argument every element of $\KK$ is also a pointwise solution to \eqref{eq:large problem in thm}. Also, $\KK$ is nonempty since $u\in \KK$. Indeed, we can repeat the proof of Lemma \ref{l:large finite} to get that $u_j\le \overline v$, hence $u\le \overline v$.
    
    Let $\{v_i\}_{i\in I}$ be a totally ordered subset of $\KK$. We will show that this set has a supremum. Define $v\coloneqq \sup_{i\in I} v_i$ and note that $v\le \overline v$. Hence, $v$ is locally bounded and in $\LLL$. Further, note that we can choose an increasing sequence $(v_n)_n$ in $\{v_i\}_{i\in I}$ such that $v=\lim_n v_n$ in $D$ by the exhaustion of $D$ argument.

    Each $v_n$ solves \eqref{eq:large problem in thm} distributionally so 
     \begin{align*}
		\int_D v_n(x) \Lo\xi(x)dx=\int_D f(v_n(x)) \xi(x)dx,\quad \xi\in C_c^\infty(D),\,n\in \N.
	\end{align*}
    By letting $n\to\infty$ in the previous inequality, since we may use the dominated convergence theorem, we get that
    \begin{align*}
		\int_D v(x) \Lo\xi(x)dx=\int_D f(v(x)) \xi(x)dx,\quad \xi\in C_c^\infty(D),
	\end{align*}
    i.e. $v$ is a distributional solution to \eqref{eq:large problem in thm}, and since it also bounded by $\overline v$, by bootstrapping Theorem \ref{t:regularity SKBM}, we have that $v$ is also a pointwise solution. Hence, by Zorn's lemma, there exists the maximal element in $\KK$.
\end{proof}

\subsection*{Acknowledgement}
This research was partly supported by the Croatian Science Foundation under the project IP-2022-10-2277.

The first-named author also acknowledges financial support under the National Recovery and Resilience Plan (NRRP), Mission 4, Component 2, Investment 1.1, Call for tender No. 104 published on 2.2.2022 by the Italian Ministry of University and Research (MUR), funded by the European Union – NextGenerationEU– Project Title “Non–Markovian Dynamics and Non-local Equations” – 202277N5H9 - CUP: D53D23005670006 - Grant Assignment Decree No. 973 adopted on June 30, 2023, by the Italian Ministry of University and Research (MUR).

	\bibliographystyle{abbrv}
	\bibliography{bibliography_large-spectral}

	\bigskip
	
	\noindent{\bf Ivan Bio\v{c}i\'c}
	
	\noindent Department of Mathematics, Faculty of Science, University of Zagreb, Zagreb, Croatia,
	
	\noindent Email: \texttt{ibiocic@math.hr}
    
        \noindent Department of Mathematics ``Giuseppe Peano", University of Turin, Turin, Italy,
	
	\noindent Email: \texttt{ivan.biocic@unito.it}
	
	\bigskip
	
	\noindent{\bf Vanja Wagner}
	
	\noindent Department of Mathematics, Faculty of Science, University of Zagreb, Zagreb, Croatia,
	
	\noindent Email: \texttt{wagner@math.hr}
\end{document}